\documentclass[12pt]{amsart}
\usepackage{amssymb} 
\newtheorem{thm}{Theorem}[section]

\newtheorem{lem}[thm]{Lemma}
\newtheorem{prop}[thm]{Proposition}

\newtheorem{rem}[thm]{Remark}

\theoremstyle{question}
\newtheorem{qu}[thm]{Question}
\newtheorem{con}[thm]{Conjecture}
\numberwithin{equation}{section}

\begin{document}
\title[Non-cyclic graph of a group]{Non-cyclic graph of a group}
\begin{center}
{\bf\Large NON-CYCLIC GRAPH OF A GROUP}
\end{center}
\vspace{1cm}
\begin{center}
{\bf A. Abdollahi$~^*$} \;\;\;  and  \;\;\; {\bf A. Mohammadi
Hassanabadi}
\\ Department of Mathematics,\\ University of Isfahan,\\ Isfahan
81746-73441,\\ Iran.
\end{center}
\thanks{$~^*$ Corresponding Author. e-mail: {\tt a.abdollahi@math.ui.ac.ir}}
\subjclass{Primary 20D60, Secondary 05C25.}%
\keywords{Non-cyclic graph; finite group.}%
\date{}%

\begin{abstract}
We associate a graph $\Gamma_G$  to a non locally cyclic group $G$
(called the non-cyclic graph of $G$) as follows: take $G\backslash
Cyc(G)$ as vertex set, where $Cyc(G)=\{x\in G \;|\;
\left<x,y\right> \; \text{is cyclic for all} \; y\in G\}$, and
join two vertices if they do not generate a cyclic subgroup. We
study the properties of this graph and we establish some graph
theoretical properties (such as regularity) of this graph in terms
of the group ones. We prove that the clique number of $\Gamma_G$
is finite if and only if $\Gamma_G$ has no infinite clique. We
prove that if $G$ is a finite nilpotent group and $H$ is a group
with $\Gamma_G\cong\Gamma_H$ and $|Cyc(G)|=|Cyc(H)|=1$, then $H$
is a finite nilpotent group.
 We give some examples of groups
$G$ whose non-cyclic graphs are ``unique'', i.e., if
$\Gamma_G\cong \Gamma_H$ for some group $H$, then $G\cong H$. In
view of these examples, we conjecture that every finite
non-abelian simple group has a unique non-cyclic graph. Also we
give some examples of finite non-cyclic groups $G$ with the
property that if $\Gamma_G \cong \Gamma_H$ for some group $H$,
then $|G|=|H|$. These suggest the question whether the latter
property holds for all finite non-cyclic groups.
\end{abstract}
\maketitle
\section{\bf Introduction and results}
Let $G$ be a group. Recall that the centralizer of an element
$x\in G$ can be defined by
$$
C_G(x)= \{y\in G| \langle x,y\rangle {~\rm is ~abelain}\}
$$
which is a subgroup of $G$. If, in the above definition, we
replace the word ``abelian'' with the word ``cyclic'' we get a
subset of the centralizer, called the {\it cyclicizer} (see
\cite{Holms1,Holms2}). To be explicit, define the cyclicizer of an
element $x\in G$, denoted by $Cyc_G(x)$, by
$$
Cyc_G (x)= \{y\in G\big\vert \langle x,y\rangle {~\rm is
~cyclic}\}.
$$
Also for a non-empty subset $X$ of $G$, we define the cyclicizer
of $X$ in $G$, to be
$$
Cyc_G(X)= \bigcap_{x\in X} Cyc_G(x),
$$
when $X=G$; we call $Cyc_G(G)$ {\it the cyclicizer} of $G$, and
denote it by $Cyc(G)$, so
$$
Cyc(G)= \{y\in G\big\vert \langle x,y\rangle {~\rm is~ cyclic~
for~ all~~ } x\in G\}
$$
It is a simple fact that for any  group $G$, $Cyc(G)$ is a locally
cyclic subgroup of $G$. As it is mentioned in \cite{Holms1}, in
general for an element $x$ of a group $G$, $Cyc_G(x)$ is not a
subgroup of $G$. For example, in the group $H=\mathbb{Z}_2\oplus
\mathbb{Z}_4$, we have
$$
Cyc_H((0,2)) = \{(0,0), (0,1), (0,2), (0,3), (1,1), (1,3)\}
$$
which is not a subgroup of $H$ (See also Theorem \ref{tt}, below).
One can associate a graph to a group in many different ways (see
for example \cite{AKM,BHM,GKNP,moghadam,N,W}). The main idea is to
study the structure of the group by the graph theoretical
properties of the associated graph. Here we consider the following
way to associate a graph with a non locally cyclic group.

Let $G$ be a non locally cyclic group. We associate a graph
$\Gamma_G$ to $G$ (called the non-cyclic graph of $G$)  with
vertex set $V(\Gamma_G)= G\setminus Cyc(G)$ and edge set
$$
E(\Gamma_G)= \{\{x,y\}\subseteq V(\Gamma_G) \big\vert \langle
x,y\rangle \; \text{is not cyclic}\}.
$$
Note that the degree of a vertex $x\in V(\Gamma_G)$ in the
non-cyclic graph $\Gamma_G$ is equal to $|G\backslash Cyc_G(x)|$.
We refer the reader to \cite{Dies} for undefined graph theoretical
concepts and to \cite{Rob}  for group theoretical ones. \\

The outline of this paper is as follows. In Section 2, we give
some results on the cyclicizers which will be used in the sequel.
In Section 3, we prove some general properties which hold for the
non-cyclic graph of a group, e.g., the non-cyclic graph of any non
locally cyclic group is always connected and its diameter is less
than or equal to 3. In Section 4, we characterize groups whose
non-cyclic graphs have no infinite clique. In fact we prove that
such groups have finite clique numbers and in contrast there are
groups $G$ whose non-cyclic graphs have no infinite independent
set and their independence numbers are not yet finite. In Section
5, we characterize finite non-cyclic groups whose non-cyclic
graphs are regular.  In Section 6, we characterize finite
non-cyclic abelian groups whose non-cyclic graphs have exactly two
kind degrees. Section 7 contains some results on groups whose
non-cyclic graphs are isomorphic. We  give some groups $G$ with
the property that if $\Gamma_G\cong\Gamma_H$ for some group $H$,
then  $|G|=|H|$. In Section 8 we give some groups $G$ whose
non-cyclic graphs are ``unique'', that is, if $\Gamma_G\cong
\Gamma_H$ for some group $H$, then $G\cong H$. It will be seen
that there are (many) groups whose non-cyclic graphs are not
unique.
\section{\bf Some properties  of cyclicizers}
 For a group $G$ and two non-empty subsets  $X$ and $Y$ of $G$, we denote by
 $Cyc_X(Y)$ the set  $\{x\in X
\;|\; \left<x,y\right> \;\;\text{is cyclic for all}\;\; y\in Y\}$.
\begin{lem}\label{12} Let $G$ be a group,  $x\in G$, and $D=Cyc_G(x)$. Then \begin{enumerate} \item $D$ is the
union of cosets of $Cyc(G)$. In particular, if $|D|<\infty$, then
$|Cyc(G)|<\infty$ and divides $|D|$.
\item $Cyc_D(D)$ is a locally cyclic subgroup of $G$ containing $x$.
\end{enumerate}
\end{lem}
\begin{proof}
(1) \; First note that the union of a chain of locally cyclic
subgroups is a locally cyclic subgroup. Thus, every element is
contained in at least one maximal locally cyclic subgroup. Now it
is easy to see that  $D$ is the union of the maximal locally
cyclic subgroups of $G$ which contain $x$. As each maximal
locally cyclic subgroup must contain  $Cyc(G)$, each of these
subgroups is a
union of cosets of $Cyc(G)$. Thus, so is $Cyc_G(x)$.\\
(2) \; It is clear that $x\in Cyc_D(D)$. Let $a,b\in Cyc_D(D)$
and suppose that $d\in D$. Now $\langle b,d \rangle=\langle c
\rangle$ for some $c\in G$. As $\langle c,x \rangle \leq \langle
b,d,x \rangle$ is cyclic, ($\langle d,x \rangle$ is cyclic and
contains $x$ so  its generator must belong to $D$) it follows
that $c\in D$.\\ Now $\langle ab^{-1},d \rangle \leq \langle
a,b,d \rangle=\langle a,c \rangle$ is cyclic. As $x\in D$,
$ab^{-1}\in D$. It follows that $ab^{-1}\in Cyc_D(D)$. It follows
that $Cyc_D(D)\leq G$. Now let $\{d_1,\dots,d_n\}\subseteq
Cyc_D(D)$. Then $\left<x,d_1\right>=\left<a_1\right>$ for some
$a_1\in G$. Since $x\in \left<a_1\right>$, $a_1\in D$. Thus
$\left<a_1,d_2\right>=\left<a_2\right>$ for some $a_2\in D$. If we
argue in this manner, then we find  an element $a_n\in D$ such
that $\langle d_1,\dots,d_n\rangle \leq \left<a_n\right>$. This
implies that $Cyc_D(D)$ is locally cyclic.
\end{proof}
\begin{prop}\label{Qn}{\rm (See \cite{Holms2})}
Let $G$ be a finite $p$-group for some prime $p$. Then
$Cyc(G)\not=1$ if and only if  $G$ is either a cyclic group or a
generalized quaternion group.
\end{prop}
\begin{proof}
Let $x$ be an element of order $p$ in $Cyc(G)$. If $A$ is a
subgroup of order $p$ of $G$, then $A=\left<a\right>$ for some
$a\in A$. Thus $H=\left<a,x\right>$ must be a cyclic $p$-group and
so $H$ has exactly one subgroup of order $p$. Therefore
$A=\left<x\right>$. It follows that $G$ has exactly one subgroup
of order $p$. Now \cite[Theorem 5.3.6]{Rob} completes the proof.
\end{proof}
\begin{lem}\label{factor}
 Let $G$ be any  group, $x\in G$,  $\overline{G}=\frac{G}{Cyc(G)}$ and $\widetilde{G}=\frac{G}{Z(G)}$.
Then
\begin{enumerate}
\item $\displaystyle Cyc_{\overline{G}}\big(xCyc(G)\big)=
\frac{Cyc_G(x)}{Cyc(G)}$. \item {\rm (See  \cite{Holms1})}
$Cyc(\overline{G})=1$. \item {\rm (See \cite{Holms1})}
$Cyc(\widetilde{G})=1$. \item If $G$ is neither torsion nor
torsion-free, then $Cyc(G)=1$.\item If $G$ is a torsion-free group
such that $Cyc(G)$ is non-trivial, then $Cyc(G)=Z(G)$. Moreover,
if  $Z(G)$ is  divisible, then $G$ is locally cyclic.
\end{enumerate}
\end{lem}
\begin{proof}
(1) \; Let $y\in G$ be such that $yCyc(G)\in
Cyc_{\overline{G}}\big(xCyc(G)\big)$. Then
$$\frac{\left<y,x\right>Cyc(G)}{Cyc(G)}$$ is
cyclic. Thus there exists an element $z\in \left<y,x\right>$ and
two elements $a_1$ and $a_2$ in $Cyc(G)$ such that $x=za_1$ and
$y=za_2$. Now since $$\left<x,y\right>=\left<za_1,za_2\right>\leq
\left<z,a_1,a_2\right>$$ and $\left<z,a_1,a_2\right>$ is  cyclic,
$y\in Cyc_G(x)$, as required.\\
(2) \; It follows from (1).\\
(3) \; It is straightforward.\\
(4) \; By hypothesis, $G$ has an element $x$ of infinite order and
a non-trivial element $y$ of finite order. If $Cyc(G)$ contains a
non-trivial element $c$ of finite order, then $\left<c,x\right>$
must be cyclic, a contradiction; and if $Cyc(G)$ contains an
element $d$ of infinite order, then $\left<y,d\right>$ must be
cyclic, a contradiction. Hence $Cyc(G)=1$. \\
(5) \; Suppose, for a contradiction, that there exists a central
element $x$ which is not in $Cyc(G)$. Then there exists $y\in G$
such that $\left<x,y\right>$ is a non-cyclic abelian group. If $a$
is a non-trivial element of $Cyc(G)$, then
$H=\left<a,x,y\right>\cong A= \mathbb{Z} \oplus \mathbb{Z}$. Since
$Cyc_A((1,0))=\left<(1,0)\right>$ and
$Cyc_A((0,1))=\left<(0,1)\right>$, $Cyc(A)=1$. It follows that
$Cyc(H)=1$, which is a contradiction, since $a\in Cyc(H)$. This
proves that $Cyc(G)=Z(G)$. Now suppose that  $Z(G)$ is divisible
and let $y$ be any element of $G$. If $a$ is a non-trivial
element of $Cyc(G)$, then $\left<a,y\right>$ is cyclic, so $y^n
\in \left<a\right>\leq Z(G)$ for some non-zero integer $n$. It
follows that $y^n=z^n$, for some $z\in Z(G)$, since $Z(G)$ is
divisible. Hence $(yz^{-1})^n=1$, and so $y\in Z(G)$. Thus
$G=Z(G)=Cyc(G)$, as required.
\end{proof}
We end this section with the following question.
\begin{qu}
Let $G$ be a torsion free group such that $Cyc(G)$ is non-trivial.
Is it true that $G$ is locally cyclic?
\end{qu}
\section{\bf Some properties of non-cyclic graph}
For a simple graph $\Gamma$, we denote by $\text{diam}(\Gamma)$
the diameter of $\Gamma$.
\begin{prop}\label{2-elem}
Let $G$ be a non locally cyclic group. Then
$\text{diam}(\Gamma_G)=1$ {\rm (}or equivalently $\Gamma_G$ is
complete{\rm )} if and only if $G$ is an elementary abelian
$2$-group.
\end{prop}
\begin{proof}
Suppose that $\text{diam}(\Gamma_G)=1$. If $x\not =x^{-1}$, for
some $x\in G\backslash Cyc(G)$, then since $\left<x,x^{-1}\right>$
is obviously cyclic, $x$ is not incident to $x^{-1}$, a
contradiction. Hence $x^2=1$ for all $x\in G\backslash Cyc(G)$. If
$z\in Cyc(G)$, then $xz\in G\backslash Cyc(G)$ for every $x\in
G\backslash Cyc(G)$. Thus $(xz)^2=1$ and since $z\in Cyc(G)\leq
Z(G)$, $x^2z^2=1$ from which it follows that $z^2=1$, since
$x^2=1$. Hence $x^2=1$ for all $x\in G$ and so $G$ is an
elementary abelian 2-group.\\

 The converse is clear.
\end{proof}
\begin{prop}\label{diam}
Let $G$ be a non locally cyclic group. Then $\Gamma_G$ is
connected and $\text{diam}(\Gamma_G)\leq 3$. Moreover, if
$Z(G)=Cyc(G)$, then $\text{diam}(\Gamma_G)=2$.
\end{prop}
\begin{proof}
Suppose that $x$ and $y$ are two vertices of $\Gamma_G$ such that
there is no path of length at least 2 between them. It follows
that $G=Cyc_G(x) \cup Cyc_G(y)$.\\
If $Z(G)=Cyc(G)$, then $x$ and $y$ are non-central elements of $G$
and
$$G=Cyc_G(x) \cup Cyc_G(y) \subseteq C_G(x) \cup C_G(y).$$ It
follows that either $G=C_G(x)$ or $G=C_G(y)$, which gives a
contradiction, since $x$ and $y$ are not central elements. Now
since $Z(G)=Cyc(G)$, $G$ is not elementary abelian, so
$\text{diam}(\Gamma_G)=2$, by Proposition \ref{2-elem}. \\
Now consider the general case. We prove that for all $$t_1\in
Cyc_G(x)\backslash Cyc_G(y)\; \text{and for all}\; t_2\in
Cyc_G(y)\backslash Cyc_G(x),$$  $t_1$ and  $t_2$ are adjacent.
Suppose, for a contradiction, that $\left<t_1,t_2\right>$ is
cyclic for some $t_1$ and $t_2$ in $Cyc_G(x)\backslash Cyc_G(y)$
and $Cyc_G(y)\backslash Cyc_G(x)$, respectively. Thus
$\left<t_1,t_2\right>=\left<t\right>$ for some $t\in G$ and so
$t\in Cyc_G(x)$ or $t\in Cyc_G(y)$. If $t\in Cyc_G(x)$, then
$\left<t,x\right>=\left<t_1,t_2,x\right>$ is cyclic, so
$\left<t_2,x\right>$ is cyclic, a contradiction. Similarly the
case $t\in Cyc_G(y)$ gives  a contradiction. Hence $x-t_2-t_1-y$
is a path of length 3 between $x$ and $y$ for all $t_1$ and $t_2$
in $Cyc_G(x)\backslash Cyc_G(y)$ and $Cyc_G(y)\backslash
Cyc_G(x)$, respectively. This completes the proof.
\end{proof}
\begin{rem} \label{rem1}\begin{enumerate}{\rm
\item Suppose that $G=Dr_{p\in T} G_p$ is a group which is the
direct product of $p$-groups $G_p$, where $T$ is a set of prime
numbers. Then $Cyc(G)=\left<Cyc(G_p) \;|\; p\in T\right>$. \item
Let $G$ be a torsion abelian group. Then $Cyc(G)$ is the subgroup
of $G$ generated by cyclic or quasicyclic primary components  of
$G$.}
\end{enumerate}
\end{rem}
\begin{lem}\label{diam-abelian}
Let  $G$ be a finite non-cyclic nilpotent group. Then \\
$\text{diam}(\Gamma_G)\leq 2$.
\end{lem}
\begin{proof}
 Let $x,y\in V(\Gamma_G)$ and
$x\not=y$.  Suppose that  $p_1,\dots,p_k$ are the prime divisors
of $|G|$ and $G_i$ is the Sylow $p_i$-subgroup  of $G$. Suppose
that $x=x_1\cdots x_k$ and $y=y_1\cdots y_k$, where $x_i,y_i\in
G_i$ for every $i\in\{1,\dots,k\}$. Suppose that $x$ is not
incident to $y$.  Since $x$ and $y$ are not in $Cyc(G)$, Remark
\ref{rem1} implies that
 there exist $i,j \in\{1,\dots,k\}$ such that $x_i\not\in Cyc(G_i)$ and $y_j \not\in
Cyc(G_j)$. Thus   there exist elements $z_i\in G_i$ and $z_j \in
G_j$ such that $x_i$ and $y_j$ are incident to $z_i$ and $z_j$,
respectively. If $i\not=j$, then since
$\left<x,z_iz_j\right>=\left<x_1,\dots,x_k,z_i,z_j\right>$ and
$\left<y,z_iz_j\right>=\left<y_1,\dots,y_k,z_i,z_j\right>$, we
have that $z_iz_j$ is incident to both $x$ and $y$. Now assume
that $i=j$, so
$\left<x_i,y_j\right>=\left<x_i,y_i\right>=\left<a\right>$ for
some $a\in G_i$, since $x$ and $y$ are not adjacent. If
$Cyc(G_i)=1$, then let $z$ be an element of order $p_i$ in
$\left<a\right>$. Then $\left<z\right> \leq \left<x_i\right> \cap
\left<y_i\right>$. Since $Cyc(G_i)=1$, there exists an element
$b\in G_i$ such that $z$ is incident to $b$. Now since
$\left<b,z\right> \leq \left<b,x\right> \cap \left<b,y\right>$,
$b$ is incident to both $x$ and $y$. So $x-b-y$ is a path of
length 2, as required. Now assume that $Cyc(G_i)\not=1$. Since
$G_i$ is a finite $p_i$-group, by Proposition \ref{Qn}, $G_i$  is
either cyclic or generalized quaternion. The former case is
false, since $x_i\not\in Cyc(G_i)$; and so $G_i$ is a generalized
quaternion group. Thus $Cyc(G_i)=Z(G_i)$ and so by Proposition
\ref{diam}, there exists an element $z\in G_i$ which is adjacent
to both $x_i$ and $y_i$. It follows that $z$ is incident to both
$x$ and $y$ and so $x-z-y$ is a path of length 2 between $x$ and
$y$, as required. Hence we have proved that the distance between
any two vertices of $\Gamma_G$ is $1$ or $2$, which completes the
proof.
\end{proof}
\begin{prop}\label{mixed}
Let $G$ be a group which is neither torsion nor torsion-free. Then
$\text{diam}(\Gamma_G)=2$.
\end{prop}
\begin{proof}
In view of Proposition \ref{2-elem}, it is enough to show that
$\text{diam}(\Gamma_G)\leq 2$. By hypothesis $G$ has an element
$x$ of infinite order and a non-trivial element $y$ of finite
order. Let $a$ and $b$ be two non-trivial elements of $G$. If $a$
is of finite order and $b$ is of infinite order or vice versa,
then $a$ is adjacent to  $b$. If $a$ and $b$ are both of finite
orders, then $a-x-b$ is a path of length two between $a$ and $b$
and if $a$ and $b$ are both of infinite order, then $a-y-b$ is a
desired path. This completes the proof.
\end{proof}
\begin{lem}\label{ZZ}
If $G=\mathbb{Z}\oplus \mathbb{Z}$, then $\text{diam}(\Gamma_G)=
2$ and $Cyc(G)=1$.
\end{lem}
\begin{proof}
It follows from Lemma \ref{factor}-(5) that $Cyc(G)=1$. To prove
$\text{diam}(\Gamma_G)= 2$ suppose that  $x$ and $y$ are two
arbitrary distinct non-trivial elements of $G$. We prove that
there exists a path of length 2 between $x$ and $y$,  if  $x$ is
not incident to $y$.  Thus $x=ta$ and $y=sa$ for some
$a=(a_1,a_2)\in G$ and for non-zero integers $t$ and $s$. If
$a_1=0$, then $x-(1,0)-y$ is a path of length two in $\Gamma_G$,
and if $a_2=0$, then $x-(0,1)-y$ is such a path, so we may assume
that $a_1$ and $a_2$ are non-zero. In this case, it is easy to see
that $x-(ta_1,sa_2)-y$ is a path of length two in $\Gamma_G$
between $x$ and $y$. Now Proposition \ref{2-elem} completes the
proof.
\end{proof}
\begin{prop}\label{t-f}
Let $G$ be a torsion-free non locally cyclic group. Then
$\text{diam}(\Gamma_G)=2$.
\end{prop}
\begin{proof}
Let $x$ and $y$ be two  distinct vertices  of $\Gamma_G$. Suppose
that $x$ and $y$ are not incident and suppose, for a
contradiction, that there is no path of length 2 between $x$ and
$y$. It follows that $G=Cyc_G(x) \cup Cyc_G(y)$ and so $G=C_G(x)
\cup C_G(y)$.
 This implies that either $x\in Z(G)$ or $y\in Z(G)$. Assume that $x\in Z(G)$. Since $x$ is
 a vertex, there exists a vertex $z$  incident to $x$.
 If $z$ is incident to $y$, then $x-z-y$ is a path of length 2.
 Thus we may assume that $z$ is not incident to $y$. It follows
 that $H=\left<x,z,y\right>\cong \mathbb{Z}\oplus \mathbb{Z}$.
 Then,
 Lemma \ref{ZZ} implies that $x$ and $y$ are vertices of $\Gamma_H$ and
   there exists a path of length 2 between $x$ and $y$. Now
   Proposition \ref{2-elem} completes the proof.
\end{proof}
\begin{prop}\label{diam-nil}
Let $G$ be a non locally cyclic group. If $G$ is locally
nilpotent, then $\text{diam}(\Gamma_G)\leq 2$.
\end{prop}
\begin{proof}
Let $x$ and $y$ be two distinct vertices of $\Gamma_G$. Then there
exist  two vertices $c_1$ and $c_2$ in $\Gamma_G$ such that $x$
and $y$ are  incident to $c_1$ and $c_2$, respectively. Let
$H:=\left<x,y,c_1,c_2\right>$. Then $H$ is a non-cyclic nilpotent
group. Since $x$ and $y$ are also two vertices in $\Gamma_H$,  it
is enough to show that there is a path of length 2 in $\Gamma_H$
between $x$ and $y$. If $H$ is neither torsion nor torsion-free,
then the proof follows from Proposition \ref{mixed}. So we may
assume that $H$ is either torsion-free or torsion. If $H$ is
torsion-free, then Proposition \ref{t-f} completes the proof. If
$H$ is torsion, since $H$ is a finitely generated nilpotent group,
 $H$ is finite. Now  Lemma \ref{diam-abelian} completes the
proof.
\end{proof}
\begin{prop}
Let $S_3$ be the symmetric group of degree $3$ and $G=\mathbb{Z}_6
\times S_3$. Then $\text{diam}(\Gamma_G)=3$.
\end{prop}
\begin{proof}
It is easy to see that the shortest path between the elements
$(3,e)$ and $(2,e)$ is of length 3, where $e$ is the identity
element of $S_3$. Hence the proof follows from Proposition
\ref{diam}.
\end{proof}
We finish this section with the following questions.
\begin{qu}\label{qu-di}
Is it possible to characterize all finite groups $G$ having the
property that  $\text{diam}(\Gamma_G)=3$?
\end{qu}
In view of Propositions \ref{diam} and \ref{diam-nil}, the finite
groups mentioned in Question \ref{qu-di} must be non-nilpotent
with non-trivial center.
\begin{qu}
What can be said about a finite non-cyclic group  $G$ whose {\em
cyclic graph} namely, the complement of $\Gamma_G$ is connected?
\end{qu}
\section{\bf Groups whose non-cyclic graphs have no infinite
clique} A subset $X$ of the vertices of a simple graph  $\Gamma$
is called a {\it clique} if the induced subgraph on $X$ is a
complete graph. The maximum size of a clique (if exists) in a
graph $\Gamma$ is called the {\it clique number} of $\Gamma$ and
is denoted by $\omega(\Gamma)$.\\

Let $G$ be a  non-abelian group and $Z(G)$ be the center of $G$.
One can associate a graph $\nabla_G$ with $G$ (called the
non-commuting graph of $G$), whose vertex set is $G\backslash
Z(G)$ and two distinct vertices are joined   if they do not
commute. Note that if $G$ is non-abelian, $\nabla_G$ is a
subgraph of $\Gamma_G$. This graph has been studied by many
people (see e.g. \cite{AKM,moghadam,N,P}). Paul Erd\"os, who was
the first to consider the non-commuting graph of a group, posed
the following problem in 1975 \cite{N}: Let $G$ be a group whose
non-commuting graph $\nabla_G$ has no infinite clique. Is it true
that the clique number of $\nabla_G$ is finite? B. H. Neumman
\cite{N} answered positively Erd\"os' question as follows.
\begin{thm}{\rm(B. H. Neumman \cite{N})}\label{BNe} The non-commuting
graph of a group $G$  has no infinite clique if and only if
$G/Z(G)$ is finite. In this case,
 the clique number of $\nabla_G$  is finite.
\end{thm}
Using  Neumman's theorem we prove the following similar result for
the non-cyclic graph.
\begin{thm} \label{P-cyc} The non-cyclic graph of a group $G$
 has no infinite clique if and only if $G/Cyc(G)$ is finite. In this case, $\omega(\Gamma_G)$ is finite.
\end{thm}
To prove this theorem we need the following lemmas. \\
For any prime number $p$, we denote by $\mathbb{Z}_{p^\infty}$
  the $p$-primary component of $\displaystyle\frac{\mathbb{Q}}{\mathbb{Z}}$.
\begin{lem} \label{l-cyc}
Let $G$ be a group whose non-cyclic graph
 has no infinite clique. Then $G$ does not contain a subgroup
 isomorphic to $\mathbb{Z}\oplus \mathbb{Z}_p$
  for any prime number
 $p$,  $\mathbb{Z}\oplus \mathbb{Z}$, or an infinite torsion
 abelian  group $B$  with $Cyc(B)=1$.
\end{lem}
\begin{proof}
It is enough to show that the non-cyclic graph of each of these
abelian groups contains an infinite clique. It is easy to see that
$$\{(p^n,1) \;|\; n\in\mathbb{N}\} \;\;\text{and}\;\;  \{(n,1) \;|\;
n\in\mathbb{N}\}$$ are infinite cliques in $\mathbb{Z}\oplus
\mathbb{Z}_p$ and  $\mathbb{Z}\oplus \mathbb{Z}$, respectively.\\
For the last part first note that  $G$ cannot contain a subgroup
isomorphic to
 $\mathbb{Z}_{p^\infty} \oplus \mathbb{Z}_p$ for any prime number
 $p$, since
  $\{(1/p^n+\mathbb{Z},1) \;|\; n\in\mathbb{N}\}$ is an infinite
  clique in $\mathbb{Z}_{p^\infty} \oplus \mathbb{Z}_p$.
 Now suppose, for a contradiction, that $G$ contains an infinite torsion
 abelian  subgroup $B$  such that $Cyc(B)=1$. By the previous part  and  Remark
 \ref{rem1}, we have that $Cyc(C)=1$ for   every primary component $C$ of $B$.
First assume that $B$ has infinitely many non-trivial primary
components.   Then for every $i\in\mathbb{N}$ there exist elements
 $a_i$ and $b_i$ in the same  primary component of $B$ such
 that $\left<a_i,b_i\right>$ is not cyclic and for any two distinct  $i,j\in\mathbb{N}$,
 $\{a_i,b_i\}$ and $\{a_j,b_j\}$ lay  in distinct
 primary components of $B$. Now define $x_1:=a_1$
 and $x_i:=b_1+\cdots+b_{i-1}+a_i$ for all $i>1$. Then it is easy
 to see that $\{x_n \;|\; n\in\mathbb{N}\}$ is an infinite clique
 in $B$, a contradiction. Thus we may assume that $B$ contains an
 infinite abelian $p$-subgroup $A$ for some prime $p$. It follows from
 \cite[Theorem 4.3.11]{Rob} that $A$ contains a subgroup
 isomorphic to a direct sum of an infinite family   $\{A_i \;|\; i\in\mathbb{N}\}$ of
 non-trivial cyclic $p$-groups. If $a_i$ is a generator of $A_i$,
 then  $\{a_i \;|\; i\in \mathbb{N}\}$ is an
 infinite clique in $\Gamma_G$, a contradiction.
This completes the proof.
\end{proof}
\begin{lem} \label{ll-cyc}
Let $G$ be a group whose non-cyclic graph
 has no infinite clique. Then every abelian subgroup of $G$
is either   torsion-free locally cyclic, or isomorphic to a direct
sum   $ \big(\oplus_{p\in T_1}\mathbb{Z}_{p^\infty}\big) \bigoplus
\big(\oplus_{p\in T_2}\mathbb{Z}_{p^{\alpha_p}}\big) \bigoplus B$,
where $T_1$ and $T_2$ are two {\rm (}possibly empty{\rm)} disjoint
sets of prime numbers, $\alpha_p\geq 0$ integers, and $B$ is a
finite abelian $(T_1\cup T_2)'$-group with  $Cyc(B)=1$.
\end{lem}
\begin{proof}
Let $A$ be an abelian subgroup of $G$. By Lemma \ref{l-cyc}, $A$
is either torsion-free or torsion. If $A$ is torsion-free, then,
since by Lemma \ref{l-cyc} it contains no subgroup isomorphic to
$\mathbb{Z}\oplus \mathbb{Z}$, $A$ must be locally cyclic. Now
assume that $A$ is torsion.  Then it follows easily from
\cite[Theorem 4.3.11]{Rob} and Lemma \ref{l-cyc} that $A$ is of
the form stated in the lemma.
\end{proof}
\noindent{\bf Proof of Theorem \ref{P-cyc}.} If $G$ is not
abelian, then $\nabla_G$ is a subgraph of $\Gamma_G$. So
$\nabla_G$ does not contain any infinite clique. Thus by Theorem
\ref{BNe}, $G/Z(G)$ is finite. Therefore, in any case ($G$ is
abelian or not), we have that $G/Z(G)$ is finite. By Lemma
\ref{ll-cyc}, $Z(G)$ is either torsion-free locally cyclic or
torsion. Suppose that $Z(G)$ is a non-trivial torsion-free locally
cyclic group. Then for every $x\in G\backslash Z(G)$,
$\left<x,Z(G)\right>$ is abelian and since it is not torsion, it
follows from Lemma \ref{ll-cyc} that $\left<x,Z(G)\right>$ is a
torsion-free locally cyclic group. Hence, in this case,  $Z(G)=Cyc(G)$ and so $G/Cyc(G)$ is finite, as required.\\
Now assume that $Z(G)$ is an infinite torsion group. Since
$G/Z(G)$ is finite, it follows that $G$ is a locally finite group.
Thus there exists  a finite subgroup $H$ of $G$ such that
$G=Z(G)H$. By Lemma \ref{ll-cyc}, there exists subgroups $Z_1,Z_2$
and $B$ of $Z(G)$ such that $Z(G)=Z_1\oplus Z_2 \oplus B$, where
$$Z_1\cong \oplus_{p\in T_1}\mathbb{Z}_{p^\infty}, Z_2\cong \oplus_{p\in T_2}\mathbb{Z}_{p^{\alpha_p}},$$
 $T_1$ and $T_2$ are two (possibly empty) disjoint sets of prime
numbers, $\alpha_p\geq 0$ integers, and $B$ is a finite abelian
$(T_1\cup T_2)'$-group with  $Cyc(B)=1$. Thus $BH$ is a finite
subgroup of $G$ and so it is a $\pi$-group for some finite set of
primes $\pi$. Let $Z_0$ be the $(T_2\backslash \pi)$-component of
$Z_2$. We prove that $Z_1\oplus Z_0\leq Cyc(G)$ from which it
follows that $G/Cyc(G)$ is finite and this completes the proof of
the ``only if'' part. Let $x$ be an arbitrary element of $G$. Then
$\left<x,Z_1\right>$ is an abelian group. Since $Z_1$ is
divisible, $\left<Z_1,x\right>=Z_1\oplus \left<y\right>$ for some
$y\in G$. Now Lemma \ref{ll-cyc} implies that $y$ must be a
$T_1'$-element, which implies  that $\left<Z_1,x\right>$ is
locally cyclic. Thus $Z_1\leq Cyc(G)$.  Now let $x\in Z_0$ and
$g\in G$. Then $g=z_1zz'bh$ for some elements $z_1\in Z_1$, $b\in
B$, $h\in H$,  $z\in Z_0$ and  $z'$ is a $(T_2 \cap \pi)$-element
of $Z_2$. Then $\left<g,x\right>$ is cyclic if and only if
$\left<zz'bh,x\right>$ is cyclic, since $z_1\in Cyc(G)$. On the
other hand $z'bh$ is a $\pi$-element and $x,z$ are $\pi'$-elements
that generate a cyclic group. It follows that
$\left<zz'bh,x\right>$ is cyclic. Therefore $Z_0\leq Cyc(G)$ and
this completes the proof as we mentioned.

 For the converse,
 let $C$ be a subset of $G$ with  $|C|>|G:Cyc(G)|$. By
pigeon-hole principal there exists two distinct elements $x,y\in
C$ which are in the same coset $aCyc(G)$, for some $a\in G$. Thus
$x=ac_1$ and $y=ac_2$ for some $c_1,c_2\in Cyc(G)$. Since
$\left<x,y\right>\leq \left<a,c_1,c_2\right>$ and this latter
subgroup is cyclic, $\left<x,y\right>$  is cyclic. This means that
$\omega(\Gamma_G)\leq |G:Cyc(G)|$.  $\;\;\;\Box$\\

A subset $X$ of the vertices of a simple graph $\Gamma$ is called
an {\it independent set} if the induced subgraph on $X$ has no
edges. The maximum size of an independent set (if exists) in a
graph $\Gamma$ is called the {\it independence number} of $\Gamma$
and denoted by $\alpha(\Gamma)$. Note that
$\alpha(\Gamma)=\omega(\Gamma^c)$, where $\Gamma^c$ is the
complement of $\Gamma$. So the following question may be posed as
dual of Theorem \ref{P-cyc}.
\begin{qu}
Let $G$ be a group whose non-cyclic graph has no infinite
independent set. Is it true that $\alpha(\Gamma_G)$ is finite?
\end{qu}
The answer of this question is negative, since  for any prime $p$,
the non-cyclic graph of the direct sum
$\bigoplus_{i\in\mathbb{N}}\mathbb{Z}_{p^i}$ has no infinite
independent set and its independence number is not finite. More
generally we have
\begin{prop} Let $G$ be a non locally cyclic group.
\begin{enumerate}
\item The non-cyclic graph of $G$ has no infinite independent set
if and only if    every abelian subgroup of $G$ is a reduced
torsion abelian group with finitely many primary components. \item
The independence number of $G$ is finite if and only if the
exponent of $G$ is finite. In this case,
$\alpha(\Gamma_G)=\max\{|x| \;:\; x\in G\}-|Cyc(G)|$.
\end{enumerate}
\end{prop}
\begin{proof}
(1) \; Suppose that $\Gamma_G$ has no infinite independent set.
Then since $G$ cannot have any infinite locally cyclic subgroup,
it follows that
 every abelian subgroup of $G$ is a reduced
torsion abelian group with finitely many primary components. Now
assume that every abelian subgroup of $G$ is a reduced torsion
abelian group with finitely many primary components and suppose,
for a contradiction, that $G$ has an infinite independent set $X$.
Then $B=\left<X\right>$, is an infinite torsion reduced abelian
group with finitely many primary components. Since $B$ has
finitely many primary components, we may assume that elements of
$X$ lay  in a $p$-primary component of $B$, for some prime $p$.
Then $B$ is a locally cyclic $p$-group, for if $y_1,\dots,y_n\in
B$, then $\left<y_1,\dots,y_n\right>=\left<x_1,\dots,x_m\right>$
for some $x_1,\dots,x_m\in X$ and  since $X$ is an independent set
of $p$-elements, $\left<x_i,x_j\right>=\left<x_i\right>$ or
$\left<x_j\right>$. It follows that
$\left<y_1,\dots,y_n\right>=\left<x_{\ell}\right>$ for some
$\ell\in\{1,\dots,m\}$. Since $X$ is infinite, $B$ is an infinite
locally cyclic $p$-group, and so $B$ is divisible, a
contradiction. This completes the proof of part (1).\\
(2) \;  If $\alpha(\Gamma_G)$ is finite, then the exponent of $G$
will be finite and less than or equal to $\alpha(\Gamma_G)$. Now
suppose that the exponent of $G$ is finite, $e$ say. Let $X$ be an
independent set of $G$. Then $A=\left<X\right>$ is an abelian
group of exponent at most $e$. Now it is easy to see that $X$ lies
in a cyclic subgroup of $A$, so $|X|\leq e$. This means that
$\alpha(\Gamma_G)\leq e$. The second part follows easily.
\end{proof}
A {\em coloring partition} for a simple graph $\Gamma$ is a
partition of vertices of $\Gamma$ whose members are independent
sets of $\Gamma$. For a positive integer $k$, we say that a graph
$\Gamma$ is {\em $k$-colorable} if it has a coloring partition
with $k$ members. The least positive integer $k$ (if exists) such
that $\Gamma$ is $k$-colorable, is called the {\em chromatic
number} of $\Gamma$ and we denote it by $\chi(\Gamma)$. It is
clear that $\omega(\Gamma)\leq \chi(\Gamma)$, for any finite graph
$\Gamma$. we prove that indeed for our non-cyclic graphs we always
 have ``the equality''.
\begin{thm}
Let $G$ be a finite non-cyclic group. Then
$\omega(\Gamma_G)=\chi(\Gamma_G)=s$, where  $s$ is the number of
maximal cyclic subgroups of $G$. Moreover $|\frac{G}{Cyc(G)}|\leq
\max\{(s-1)^2 (s-3)!, (s-2)^3(s-3)!\}$.
\end{thm}
\begin{proof}
Let $A_1,\dots,A_s$ be all maximal cyclic subgroups of $G$ and let
$A_i=\left<a_i\right>$ for $i\in\{1,\dots,s\}$. Since
$\left<a_i,a_j\right>$ is not cyclic for distinct $i$ and $j$,
$\{a_1,\dots,a_s\}$ is a clique in $\Gamma_G$. Thus $s\leq
\omega(\Gamma_G)$. On the other hand, since every element of $G$
is contained in $A_i$ for some $i$, we have that $G=\cup_{i=1}^s
A_i$. If $A_i\backslash Cyc(G)$ is not empty, then it is an
independent set for $\Gamma_G$, so it follows that
$\chi(\Gamma_G)\leq s$ and as we mentioned $\omega(\Gamma_G)\leq
\chi(\Gamma_G)$. This completes the proof of the first part.\\
We have that $G=\cup_{i=1}^s A_i$ is an irredundant covering for
$G$ (see \cite{Tom} for definitions) and $\cap_{i=1}^s A_i$ is
clearly contained in $Cyc(G)$. Thus by the main theorem of
\cite{Tom}, we have $|\frac{G}{Cyc(G)}|\leq \max\{(s-1)^2 (s-3)!,
(s-2)^3(s-3)!\}$.
\end{proof}

\section{\bf Finite groups with regular non-cyclic graphs}
In this section we  prove the following.
\begin{thm} \label{A} Let $G$ be a non-cyclic finite group. Then the non-cyclic
graph of $G$ is regular if and only if $G$ is isomorphic to one of
the following groups:
\begin{enumerate}
\item $Q_8 \times \mathbb{Z}_n$, where $n$ is an odd integer and
$Q_8$ is the quaternion group of order $8$ . \item  $P\times
\mathbb{Z}_m$, where $P$ is a finite non-cyclic group of  prime
exponent $p$ and $m>0$ is an integer such that $\gcd(m,p)=1$.
\end{enumerate}
\end{thm}
Throughout this section let $G$  be a finite non-cyclic group
whose non-cyclic graph is regular. Thus
$|G|-|Cyc_G(x)|=|G|-|Cyc_G(y)|$ for all $x,y \in G\backslash
Cyc(G)$, and so $|Cyc_G(x)|=|Cyc_G(y)|$ for all $x,y\in
G\backslash Cyc(G)$. Note that, by Lemma \ref{factor}, the
non-cyclic graph of $\displaystyle H=\frac{G}{Cyc(G)}$ is also
regular and $Cyc(H)=1$. Therefore we have $|Cyc_H(x)|=|Cyc_H(y)|$
for any two  non-trivial elements $x,y$ of $H$.
\begin{lem}\label{referee-lem}
For any two non-trivial elements $x, y\in H$, either $Cyc_H(x)\cap
Cyc_H(y)=1$ or $Cyc_H(x)=Cyc_H(y)$ and $Cyc_H(x)$ is a subgroup
of $H$.
\end{lem}
\begin{proof}
First note that if $\langle m\rangle$ is a maximal cyclic
subgroup of $H$ such that $x\in \langle m\rangle$, then
$Cyc_H(m)=\left<m\right>$ and $Cyc_H(x)$ contains $\langle m
\rangle$. Since all the cyclicizers of non-trivial elements have
the same size, it follows that $Cyc_H(x)=\langle m\rangle$. Thus,
for all $1\not= x\in H$, $Cyc_H(x)$ is the unique maximal cyclic
subgroup which contains $x$.\\
Therefore, if $1\not=z\in Cyc_H(x) \cap Cyc_H(y)$, then $z$
belongs to the unique maximal cyclic subgroup containing $x$ and
to the unique  maximal cyclic subgroup containing $y$. Since $z$
is contained in a unique maximal cyclic subgroup, it follows that
$Cyc_H(x)=Cyc_H(y)$.
\end{proof}
 In the proof of Theorem \ref{A}, we use the following result due to I.M.
Isaacs \cite{Isac} on equally partitioned groups.
\begin{thm} \label{partition} {\rm (see I.M. Isaacs \cite{Isac})}
Let $A$ be a finite non-trivial  group and let $n>1$ be an integer
such that $\{A_i \;|\; i=1,\dots,n\}$ is a set of subgroups of $A$
with the property that $A=\cup_{i=1}^n A_i$, $|A_i|=|A_j|$ and
$A_i\cap A_j =1$ for any two distinct $i,j$. Then $A$ is a group
of prime exponent.
\end{thm}
\noindent {\bf Proof of Theorem \ref{A}.} Since $H=\bigcup_{x\in
H\backslash\{1\}}Cyc_H(x)$, it follows from Lemma
\ref{referee-lem} and Theorem \ref{partition} that $H$ is a finite
group of exponent $p$ for some prime $p$. Assume that
$Cyc(G)=\langle x \rangle\times \langle y\rangle$, where $x$ is a
$p'$-element and $y$ is a $p$-element. Since $Cyc(G)\leq Z(G)$
and $H$ is a $p$-group, $G$ is nilpotent and $G=P \times \langle
x\rangle$, where $P$ is the Sylow $p$-subgroup of $G$. Note that
$y\in Cyc(P)$. If $y=1$, then $P$ is  a $p$-group of exponent $p$
and if $y\not=1$, then $P$ contains exactly one subgroup of order
$p$. Thus $P$ is a generalized quaternion group or a cyclic
group, by Proposition \ref{Qn}. But if $P$ is a generalized
quaternion group of order greater than $8$, then the exponent of
$H$ is greater than 2 which is not possible. Thus $P$ must be
isomorphic to $Q_8$, the quaternion group of order 8. This
completes the proof of the ``only if'' part.  The converse is
straightforward.
\;\;\;$\Box$\\

Let $K$ be a group. There are cases in which $Cyc_K(x)$ is a
subgroup of $K$ for every $x\in K$; following \cite{Holms2}, call
such groups {\it tidy}. For example, if $K$ is a group such that
every non-identity element of $K$ is of prime order then $K$ is a
tidy group. The next result shows that Theorem \ref{A} cannot be
improved in the following sense:
 If  $K$ is a finite non-cyclic group whose  non-cyclic
 graph {\em has two kind degrees}, then $K$ is not necessarily a tidy group,
 where for a given positive integer $k$,
we say  that a  simple graph {\em has $k$ kind degrees} if the
size of the set of degrees of
 vertices is $k$. Note that a regular graph is a graph of one kind
 degree.
 \begin{thm}\label{tt}
Let $p$ be a prime number,  $m\geq 1$ and $n>1$ be positive
integers. Let $G=\bigoplus_{i=1}^n\mathbb{Z}_{p^m}$. Then
$Cyc(G)=\left<(0,0,\dots,0)\right>$ and if $x\in G$ such that
$|x|=p^{\ell}$ with $1\leq \ell < m$, then
$$|Cyc_G(x)|=p^{\ell}+\sum_{i=\ell+1}^m
\big(p^{i}-p^{i-1} \big)p^{(n-1)(m-i+\ell)},\eqno{(*)}$$ and
$|Cyc_G(y)|=p^m$ for every element $y$ of order $p^m$. It follows
that $\Gamma_G$ is a graph having $m$ kind degrees and
 $G$ is not a tidy group.
 \end{thm}
 \begin{proof}
 Note that if $|y|=p^m$, then $Cyc_G(y)=\left<y\right>$, since
 $|y|$ is equal to the exponent of $G$.
Now suppose that $x'$ is another element of $G$ of order $p^\ell$,
we first prove  that   $|Cyc_G(x)|=|Cyc_G(x')|$ and next, it is
enough to prove that $|Cyc_G((p^{m-\ell},0,\dots,0))|$ is the
number of the right hand side of $(*)$. There exist elements $z_1$
and $z_1'$ in $G$ of order $p^m$ such that $p^{m-\ell}z_1=x$ and
$p^{m-\ell}z_1'=x'$ (we write $G$ additively). On the other hand,
there exist elements $z_2,\dots,z_{n}$ and $z'_2,\dots,z'_{n}$ in
$G$ such that the sets $\{z_1,z_2,\dots,z_{n}\}$ and
$\{z'_1,z'_2,\dots,z'_{n}\}$ are both {\em linearly independent}
in the sense of \cite[pp. 95-96]{Rob}, and
$G=\left<z_1,z_2,\dots,z_{n}\right>=\left<z_1',z'_2,\dots,z'_{n}\right>$.
Now define $\alpha :G\rightarrow G$ by
$$\alpha\big(\sum_{i=1}^{n}{m_iz_i}\big)=\sum_{i=1}^{n}{m_iz'_i} \;\; \text{for
all} \;\; m_i\in \{0,1,\dots,p^m-1\}.$$  It is easy to see that
$\alpha$ is an automorphism of $G$ with the property that
$\alpha(x)=x'$. Thus $\alpha: Cyc_G(x)\rightarrow Cyc_G(x')$ is a
bijection and so $|Cyc_G(x)|=|Cyc_G(x')|$. Now  if
$$(a_1,a_2,\dots,a_n)\in
Cyc_G((p^{m-\ell},0,\dots,0))\backslash\left<(p^{m-\ell},0,\dots,0)\right>,$$
then $\left<(a_1,a_2,\dots,a_n),(p^{m-\ell},0,\dots,0)\right>$ is
cyclic. It follows that $(a_1,a_2,\dots,a_n)$ is of order $p^i$
for some $i>\ell$ and so
$$(p^{m-\ell},0,\dots,0)=p^{i-\ell}t(a_1,a_2,\dots,a_n)$$ for some
$t\in\{0,1,\dots,p^m-1\}$ such that $\gcd(t,p)=1$. Therefore
$p^{i-\ell}ta_1\overset{p^m}{\equiv} p^{m-\ell}$ and
$p^{i-\ell}ta_j\overset{p^m}{\equiv} 0$ for each $1<j\leq n$.
Since $\gcd(t,p)=1$, we have $p^{i-\ell}a_j\overset{p^m}{\equiv}
0$, so
$$a_j\in\{0,p^{i-\ell},2p^{i-\ell},\dots,(p^{m-i+\ell}-1)p^{i-\ell}\} \;\;\text{for each}\;\; j>1.$$
Now let $a_1=p^ks$ where $\gcd(p,s)=1$ and
$a_1\in\{0,1,\dots,p^m-1\}$. Thus
$p^{i-\ell+k}st\overset{p^m}{\equiv} p^{m-\ell}$ from which it
follows that $k=m-i$, since $\gcd(st,p)=1$. Hence
$$a_1\in\big\{p^{m-i}s \;|\; s\in\{1,\dots,p^i-1\}
\;\;\text{and}\;\; \gcd(s,p)=1\big\}.$$ Therefore
\begin{align*}
\big|Cyc_G((p^{m-\ell},0,\dots,0))\big|=
\big|\left<(p^{m-\ell},0,\dots,0)\right>\big|
  +  \big|\big\{(a_1,a_2,\dots,a_n)\in G \;:\; &\\
|(a_1,a_2,\dots,a_n)|=p^i  \;\;\text{for some}\;\; \ell<i\leq m
\;\;\;\text{and}\;\;\; &\\
(p^{m-\ell},0,\dots,0)\in\left<(a_1,a_2,\dots,a_n)\right>\big\}\big|&\\=p^{\ell}+\sum_{i=\ell+1}^m
\big(p^{i}-p^{i-1} \big)p^{(n-1)(m-i+\ell)}.
\end{align*}
This completes the proof.
 \end{proof}

\section{\bf Finite groups whose non-cyclic graphs have  two kind degrees}

Throughout this section let $G$ be a finite non-cyclic group whose
non-cyclic graph has two kind degrees and assume that
$H=\frac{G}{Cyc(G)}$. Then by Lemma \ref{factor}, $\Gamma_H$ has
also two kind degrees.
\begin{thm}\label{k=2}
If $G$ is  nilpotent, then   $H$ is a $p$-group for some prime
$p$.
\end{thm}
\begin{proof}
Suppose, for a contradiction, that $|H|$ is divisible by at least
two distinct prime numbers $p_1$ and $p_2$.  Thus $H=P_1\times P_2
\times P_3$, where $P_3$ is a Hall $\{p_1,p_2\}'$-subgroup of $H$.
If $x\in P_i$, then $|Cyc_H(x)|=|Cyc_{P_i}(x)|\Pi_{j\not=i}|P_j|$.
Note that since $Cyc(H)=1$ and $$Cyc(H)=Cyc(P_1)\times
Cyc(P_2)\times Cyc(P_3),$$ $Cyc(P_i)=1$ for every
$i\in\{1,2,3\}$.  It follows that
$$|Cyc_H(x_1)|,|Cyc_H(x_2)| \;\;\text{and}\;\; |Cyc_H(x_3)|$$ are all distinct
for every $x_i\in P_i\backslash \{1\}$.  This implies that
$\Gamma_H$ has at least 3  kind degrees, if $P_3\not=1$. Therefore
$P_3=1$. Now take $x_i\in P_i$ such that $|x_i|=exp(P_i)$ for
$i=1,2$. Then $Cyc_H(x_1x_2)=\left<x_1x_2\right>$, because
$exp(H)=|x_1x_2|$. Since $Cyc(P_i)=1$, we have that
$|P_i|>exp(P_i)$ for $i=1,2$. Hence $|Cyc_H(x_1)|$, $|Cyc_H(x_2)|$
and $|Cyc_H(x_1x_2)|$ are all distinct. This contradiction
completes the proof.
\end{proof}
\begin{thm}
The group $G$ is abelian if and only if  $G\cong \mathbb{Z}_m
\bigoplus \big(\oplus_{i=1}^n {\mathbb{Z}}_{p^2}\big)$, where $p$
is a prime number, $n>1$ is an integer and $m$ is a positive
integer  with $\gcd(m,p)=1$.
\end{thm}
\begin{proof}
By Theorem \ref{k=2},  $H$ is an abelian $p$-group for some prime
$p$. Thus $H\cong
M=\bigoplus_{i=1}^{\ell}\big(\oplus_{j=1}^{t_i}{\mathbb{Z}_{p^{\alpha_i}}}\big)$
for some positive integers $t_i$, $\ell$ and $\alpha_1<\cdots<
\alpha_\ell$. If we prove $\ell=1$, then Theorem \ref{tt} implies
that $\alpha_1=2$, and since $G$ is abelian, by Remark \ref{rem1}
we have that $G\cong Cyc(G) \oplus H$. Therefore it is enough to
show that $\ell=1$. Suppose, for a contradiction, that $\ell>1$.
Assume first $\ell>2$ and take $x_i$ to be a generator of a direct
summand $\mathbb{Z}_{p^{\alpha_i}}$ for $i=1,2,3$. Then
$Cyc_M(x_i)=\left<x_i\right>$ for $i=1,2,3$. Since
$\alpha_1<\alpha_2<\alpha_3$, we have  $|x_1|<|x_2|<|x_3|$, which
gives a contradiction. Thus we may assume that $\ell=2$. Hence
$$H\cong M=\big(\oplus_{j=1}^{t_1}{\mathbb{Z}_{p^{\alpha_1}}}\big) \bigoplus
\big(\oplus_{j=1}^{t_2}{\mathbb{Z}_{p^{\alpha_2}}}\big).$$ Now
take $x_i$ to be a generator of a direct summand
$\mathbb{Z}_{p^{\alpha_i}}$ for $i=1,2$. Then, since
$\alpha_1<\alpha_2$,  $\left<x_2\right> \cup \left<x_1x_2\right>
\subseteq Cyc_M(x_2^{p^{\alpha_1}}).$ Since $\left<x_2\right> \cap
\left<x_1x_2\right>=\left<x_2^{p^{\alpha_1}}\right>$, we have
$|Cyc_M(x_2^{p^{\alpha_1}})|\geq 2p^{\alpha_2}-p^{\alpha_1}$.
Hence $|Cyc_M(x_2^{p^{\alpha_1}})|$, $|Cyc_M(x_1)|$ and
$|Cyc_M(x_2)|$ are all distinct, a contradiction. Therefore
$\ell=1$.\\ The converse follows from Theorem \ref{tt}, Lemma
\ref{factor} and Remark \ref{rem1}.
\end{proof}
\section{\bf Groups with the same non-cyclic graph\label{a}}
In this section the following question is of  interest. What is
the relation between two non locally cyclic groups $G$ and $H$ if
$\Gamma_G\cong \Gamma_H$?\\
More precisely one may propose the following question:
\begin{qu}\label{mainqu}  For which group property $\mathcal{P}$ if
$G$ and $H$ are two non locally cyclic groups such that
$\Gamma_G\cong\Gamma_H$, and  $G$ has the group property
$\mathcal{P}$, then  $H$ has also $\mathcal{P}$?
\end{qu}
First we consider Question \ref{mainqu}, when $\mathcal{P}$ is the
property of being finite. In this case   Question \ref{mainqu} has
positive answer, as we see
\begin{prop}\label{fi}
Let $G$ be a finite non-cyclic group such that $\Gamma_G\cong
\Gamma_H$ for some group $H$. Then $H$ is a finite non-cyclic
group. Moreover $|Cyc(H)|$ divides
$$\gcd\big(|G|-|Cyc_G(g)|,|G|-|Cyc(G)| \;:\; g\in G\backslash
Cyc(G) \big).$$
\end{prop}
\begin{proof}
By the hypothesis $|H\backslash Cyc(H)|=|G|-|Cyc(G)|\not=0$. As
$Cyc(H)\leq H$ and $H\backslash Cyc(H)$ is a finite set, we have
that $H$ is finite and since  $H\backslash Cyc(H)$ is non-empty,
$H$ is non-cyclic, as required. Since $\Gamma_G\cong \Gamma_H$, we
have
$$\{deg(v), |V(\Gamma_G)| \;:\; v\in V(\Gamma_G)
\}=\{deg(w), |V(\Gamma_H)| \;:\; w\in V(\Gamma_H) \}.$$ But
$deg(w)=|H|-|Cyc_H(w)|$ for every $w\in V(\Gamma_H)$ and so it
follows from Lemma \ref{12} that $|Cyc(H)|$ divides $deg(w)$ for
every $w\in V(\Gamma_H)$ and clearly $|Cyc(H)|$ divides
$|H|-|Cyc(H)|=|V(\Gamma_H)|$. This completes the proof.
\end{proof}
When $\mathcal{P}$ is the property of being nilpotent we have no
negative answer to Question \ref{mainqu}; however  under an
additional condition we give a positive answer. In fact we have
\begin{thm}
Let $G$ be a finite non-cyclic nilpotent group. If $H$ is a group
such that $|Cyc(G)|=|Cyc(H)|=1$ and $\Gamma_G\cong \Gamma_H$,
then $H$ is a finite nilpotent group. Furthermore, for every prime
number $p$, if $P$ and $Q$ are Sylow $p$-subgroups of $G$ and $H$,
respectively,
 then $\Gamma_P \cong \Gamma_Q$.
\end{thm}
\begin{proof} First note that by Proposition \ref{fi}, $H$ is
finite. Let $p$ and $q$ be two distinct prime divisors of $|H|$.
It is enough to show that every  $p$-element $x$ of $H$ commutes
with every  $q$-element $y$ of $H$, or equivalently
$\left<x,y\right>$ is cyclic. Since for every prime divisor $r$
of $|H|$, each $r$-element of $H$ is contained in a cyclic
$r$-subgroup of $H$ which is maximal among cyclic $r$-subgroups
of $H$, we may assume that $\left<x\right>$ and $\left<y\right>$
are maximal cyclic $p$-subgroup ($q$-subgroup, respectively) of
$H$. Now let $\phi$ be a graph isomorphism from $\Gamma_H$ to
$\Gamma_G$ and note that, since $|Cyc(G)|=|Cyc(H)|=1$, we may
consider $\phi$ as a bijection  from $H$ onto $G$, by defining
$\phi(1_G)=1_H$. Now if we show that $\phi(x)$ and $\phi(y)$ are
$p$-element and $q$-element in $G$, respectively, then as $G$ is
nilpotent, $\phi(x)$ and $\phi(y)$ are not adjacent, and so $x$
and $y$ are not adjacent, as required. So finally, by the
symmetry between $x$ and $y$, it is enough to show  that $\phi(x)$
is a $p$-element of $G$.\\
First we show that, for any  $p$-element of $a$ of $G$, $\psi(a)$
is a $p$-element of $H$, where $\psi$ is the inverse of $\phi$.
Let $D=Cyc_G(a)$. Then as $G$ is the direct product of its Sylow
subgroups, $$D=Cyc_{P_i}(a)\times
\prod_{\overset{\ell\not=i}{\ell=1}}^k P_\ell,$$ where $P_{\ell}$
is the Sylow $p_{\ell}$-subgroup of $G$ and $p=p_i$. Now since
$Cyc(G)=1$, $Cyc(P_\ell)=1$ for every $\ell\in \{1,\dots,k\}$. It
follows that $Cyc_D(D)=Cyc_C(C)$, where $C=Cyc_{P_i}(a)$. Now
Lemma \ref{12} shows that $Cyc_D(D)$ is a cyclic $p$-subgroup of
$G$. Since $\psi$ is a graph isomorphism from $\Gamma_G$ to
$\Gamma_H$, we have that $\psi\big(Cyc_D(D)\big)=Cyc_{D'}(D')$,
where $D'=Cyc_H(\psi(a))$. Thus $|Cyc_{D'}(D')|$ is a $p$-power
number and so, by Lemma \ref{12}, $\psi(a)$ is a $p$-element of
$H$.\\
Now suppose that $\phi(x)=x_1x_2\cdots x_k$, where $x_{\ell}$ is a
$p_\ell$-element of $G$. Let $M=Cyc_G(\phi(x))$. Then
$M=\prod_{\ell=1}^k Cyc_{P_\ell}(x_\ell)$. Now we prove that
$$Cyc_{P_i}(x_i)=\phi\big(\left<x\right>\big).\eqno{(*)}$$ Let $z\in
Cyc_{P_i}(x_i)$. Then $\left<z,\phi(x)\right>$ is cyclic and so
$A=\left<\psi(z),x\right>$ is cyclic. By the previous part
$\psi(z)$ is a $p$-element and so $A$ is a cyclic $p$-subgroup of
$H$ containing $x$. Now as $\left<x\right>$ is a maximal cyclic
$p$-subgroup of $H$,  $\psi(z)\in \left<x\right>$ and so $z\in
\phi\big(\left<x\right>\big)$. On the other hand, by Lemma
\ref{12}, $|x|$ divides
$$|Cyc_{\psi(M)}(\psi(M))|=|Cyc_M(M)|=\Pi_{\ell=1}^k |Cyc_{D_{\ell}}(D_\ell)|,$$
where $D_{\ell}=Cyc_{P_\ell}(x_\ell)$ for every
$\ell\in\{1,\dots,k\}$. Now since $|x|$ is a $p$-power number,
Lemma \ref{12} implies that $|x|$ must divide
$|Cyc_{D_i}(D_i)|\leq |D_i|$. Therefore $|x|\leq
|Cyc_{P_i}(x_i)|$ and also we have $Cyc_{P_i}(x_i)\subseteq
\phi\big(\left<x\right>\big)$. Now since
$|\phi\big(\left<x\right>\big)|=|x|$, we have proved $(*)$. This
implies that $\phi(x)\in Cyc_{P_i}(x_i)\subseteq P_i$ and so
$\phi(x)$ is a $p$-element. This completes the proof of the first part.\\
By the proof of the first part, $H$ and $G$ are both finite
nilpotent of the same size and $$|Cyc(P_\ell)|=|Cyc(Q_\ell)|=1
\;\;\text{and}\;\; |P_\ell|=|Q_{\ell}|,$$ where $Q_{\ell}$ is the
Sylow $p_{\ell}$-subgroup of $H$.  By the proof of the previous
part we have that,  $\psi(P_\ell)$ is the Sylow $p_\ell$-subgroup
of $H$, for we have proved that every $p_{\ell}$-element of $G$
maps under $\psi$ to a $p_{\ell}$-element of $H$, and since $H$ is
now nilpotent, every $p_{\ell}$-element of $H$ maps under $\phi$
to a $p_{\ell}$-element of $G$. Hence $\phi$ induces a graph
isomorphism from $\Gamma_{P_\ell}$ to the non-cyclic graph of
$Q_\ell$, the Sylow $p_\ell$-subgroup of $H$. This completes the
proof.
\end{proof}

Now let us  consider Question \ref{mainqu}, when $\mathcal{P}$ is
the property of being a fixed finite size, that is, in view of
Proposition \ref{fi}, we are asking the following question:
\begin{qu}\label{order} Let  $G$ and $H$ be
two non-cyclic finite groups such that $\Gamma_G\cong \Gamma_H$.
Is it true that $|G|=|H|$?
\end{qu}
We were unable to find a negative answer for Question \ref{order}.
In the following we give certain groups $G$ for which Question
\ref{order} has positive answer. First let us examine some cases
in which we can use Theorem \ref{A}.
\begin{prop}\label{6} Let $n>0$ be an odd integer  and let $G$ be a group
such that  its non-cyclic graph is isomorphic to the non-cyclic
graph of  $H=Q_8\times \mathbb{Z}_n$. Then   $G\cong Q_8 \times
\mathbb{Z}_n$.
\end{prop}
\begin{proof} We have $$|H|-|Cyc(H)|=
|G|-|Cyc(G)|. \eqno({\rm I})$$ By hypothesis $\Gamma_G$ is also
regular and so by Theorem \ref{A}, $G$ is isomorphic to $G_1=Q_8
\times \mathbb{Z}_{n'}$ for some odd integer $n'$ or $G_2=P\times
\mathbb{Z}_m$, where $P$ is a finite non-cyclic group of exponent
$p$ for some prime $p$ and $m>0$ is an integer such that
$$\gcd(p,m)=1. \eqno{\rm(II)}$$ If $G\cong G_1$,  then
$|Cyc(G)|=2n'$ and by (I) we have $8n-2n=8n'-2n'$ and so $n=n'$,
thus $G\cong Q_8 \times \mathbb{Z}_n$. Now suppose, for a
contradiction, that $G\cong G_2$. Then $|Cyc(G)|=m$ and
$|Cyc_G(x)|=pm$ for all $x\in G\backslash Cyc(G)$. Therefore by
(I) we have $$6n=|P|m-m \eqno(1)$$ and by the hypothesis there
exists an element $a\in H\backslash Cyc(H)$ such that
$|Cyc_H(a)|-|Cyc(H)|=|Cyc_G(x)|-|Cyc(G)|$. Therefore $$4n-2n=pm-m.
\eqno(2) $$ Now it follows from (1) and (2) that
$\frac{|P|-1}{p-1}=3$. This yields that $p=2$, and then (2)
implies that $m$ is even which contradicts (II). This completes
the proof.
\end{proof}
\begin{prop} \label{goor}
Let $p$ and $q$ be two prime numbers and $m$ and $t$ be positive
integers with $\gcd(p,m)=\gcd(q,t)=1$. Suppose that $G=P\times
\mathbb{Z}_n$ and $H=Q\times \mathbb{Z}_t$, where $P$ and $Q$ are
finite non-cyclic groups of exponents $p$ and $q$, respectively
and $|P|=p^m$ and $|Q|=q^s$ for some  integers $m>1$ and $s>1$.
Then $\Gamma_G\cong \Gamma_H$ if and only if
$$\frac{p^m-1}{p-1}=\frac{q^s-1}{q-1} \;\; \text{and} \;\;
n(p-1)=t(q-1). \eqno{(G)}$$
\end{prop}
\begin{proof}
It is easy to see that   $\Gamma_G$ ($\Gamma_H$, respectively) is
a complete $\big(\frac{p^m-1}{p-1}\big)$-partite graph
($\big(\frac{q^s-1}{q-1}\big)$-partite graph, respectively) whose
parts have equal size $(p-1)n$ ($(q-1)t$, respectively). These
completes the proof.
\end{proof}
\begin{prop}
Let $G$ be a group. Then for a positive odd integer $n$ and an
integer $m>1$, $\Gamma_G\cong
\Gamma_{\big(\oplus_{i=1}^m\mathbb{Z}_2\big) \bigoplus
\mathbb{Z}_n}$ if and only if   $G\cong
\big(\oplus_{i=1}^m\mathbb{Z}_2\big) \bigoplus \mathbb{Z}_n$.
\end{prop}
\begin{proof}
Since the non-cyclic graph of
$\big(\oplus_{i=1}^m\mathbb{Z}_2\big) \bigoplus \mathbb{Z}_n$ is
regular,  $\Gamma_G$ is also regular. Now Theorem \ref{A} and
Proposition \ref{6} imply  that   $G\cong Q \times \mathbb{Z}_t$,
where $Q$ is a finite non-cyclic group of exponent $q$ and of
order $q^s$, for some prime number $q$ and integer $s>1$, and
$t>0$ is an integer such that $\gcd(q,t)=1$. Now Proposition
\ref{goor} implies that
$$2^m-1=\frac{q^s-1}{q-1} \;\;\text{and}\;\; n=t(q-1).
$$ Since $n$ is odd, the equality $n=t(q-1)$ yields
that $q$ cannot be odd and so $q=2$. Therefore $n=t$, $p=q=2$ and
$m=s$. Since the exponent of $Q$ is 2 and $|Q|=2^s=2^m$, $Q$ is
isomorphic to
$\oplus_{i=1}^m\mathbb{Z}_2$. This completes the proof of the ``only if'' part.\\
The converse follows from Proposition \ref{goor}.
\end{proof}
\begin{prop} Let $G$ be a group. Then for a prime number $p$  and an integer $n>0$ such that $\gcd(n,p)=1$,
$\Gamma_G \cong \displaystyle\Gamma_{\mathbb{Z}_p\oplus
\mathbb{Z}_p\oplus \mathbb{Z}_n}$ if and only if $G\cong
\mathbb{Z}_p\oplus \mathbb{Z}_p\oplus \mathbb{Z}_n$.
\end{prop}
\begin{proof} Since the non-cyclic graph of $\mathbb{Z}_p \oplus
\mathbb{Z}_p \oplus \mathbb{Z}_n$ is regular, so is $\Gamma_G$.
Now Theorem \ref{A} and Proposition \ref{6} imply  that   $G\cong
Q \times \mathbb{Z}_t$, where $Q$ is a finite group of exponent
$q$ and of order $q^s$, for some prime number $q$ and integer
$s>1$, and $t>0$ is an integer such that $\gcd(q,t)=1$. Now
Proposition \ref{goor} implies that
$$\frac{p^2-1}{p-1}=\frac{q^s-1}{q-1} \eqno{(1)}$$ {and}
$n(p-1)=t(q-1)$. From (1) it follows that $p=q(q^{s-2}+\dots+1)$
which gives $p=q$ and $s=2$, since $p$ and $q$ are prime numbers.
Therefore $n=t$ and this completes the
proof of ``only if'' part.\\
The converse follows from Proposition \ref{goor}.
\end{proof}
\begin{rem}{\rm
It was conjectured by Goormaghtigh \cite{goor} that the
Diophantine equation
$$\frac{x^m - 1}{x-1}=\frac{y^n - 1}{y - 1}, \eqno{(GO)}$$ for $x,y,m,n\in
\mathbb{N}$, $x>y>1$ and $n>m>2$ has only the solutions
$(x,y,m,n)=(5,2,3,5)$ and $(90,2,3,13)$. This conjecture has not
been solved  so far. By Proposition \ref{goor}, Question
\ref{order} is true for finite groups with regular non-cyclic
graphs, if the equation $(G)$ which is  a certain case of
Goormaghtigh's conjecture has no solution. This shows that how
``hard'' it is to settle Question \ref{order} for certain groups,
such as elementary abelian groups. In the following we use a
result  of M. Le \cite{Le}  on this conjecture to prove another
certain case of Question
 \ref{order}.}
\end{rem}
\begin{thm}\label{Le}{\rm (See M. Le \cite[Theorem]{Le})}\\
If $(x,y,3,n)\not\in\{(5,2,3,5),(90,2,3,13)\}$ is a solution of
the equation $(GO)$ with $m=3$, then we have $\gcd(x,y)>1$ and
$y\nmid x$.
\end{thm}
\begin{prop}\label{p^3}
Let $G$ be a group. Then for a prime number $p>2$  and an integer
$n>0$ such that  $\gcd(n,p)=1$, $\Gamma_G\cong
\displaystyle\Gamma_{\mathbb{Z}_p\oplus \mathbb{Z}_p \oplus
\mathbb{Z}_p\oplus \mathbb{Z}_n}$ if and only if $G\cong
\mathbb{Z}_p\oplus\mathbb{Z}_p\oplus \mathbb{Z}_p\oplus
\mathbb{Z}_n$ or $T \times \mathbb{Z}_n$, where $T$ is the only
non-abelian   group of order $p^3$ and exponent $p$.
\end{prop}
\begin{proof}
 Note that  $\Gamma_G$ is regular. Now Theorem
\ref{A} and Proposition \ref{6} imply  that   $G\cong Q \times
\mathbb{Z}_t$, where $Q$ is a finite group of exponent $q$ and of
order $q^s$, for some prime number $q$ and integer $s>1$, and
$t>0$ is an integer such that $\gcd(q,t)=1$. It follows from
Proposition  \ref{goor}  that
$$\frac{p^3-1}{p-1}=\frac{q^s-1}{q-1} \eqno{(1)}$$ {and}
$n(p-1)=t(q-1)$. Now Theorem \ref{Le} implies that $p=q$ and
$s=3$. Therefore $n=t$ and $Q$ is a finite group of exponent $p$
and of order $p^3$. This completes the proof of the ``only if'' part.\\
The converse follows from Proposition  \ref{goor}.
\end{proof}
\begin{rem}\label{rem-is}{\rm
(1) \; Proposition  \ref{p^3} shows that Question \ref{mainqu} is
not
true for the group property of being abelian.\\
(2) \; Proposition  \ref{p^3} shows that Question \ref{mainqu} is
not true for the group property of being isomorphic to a fixed
group and by Proposition  \ref{goor} one can find further examples
of non-isomorphic groups whose non-cyclic graphs are isomorphic.}
\end{rem}
We end this section with the following proposition which may be
considered  as an starting point toward settling Question
\ref{order} for abelian $p$-groups with non-regular non-cyclic
graphs.
\begin{prop}\label{p-p^2}
Let $p$ be a prime number and  $G$ be a finite non-cyclic
$p$-group of order $p^n$ for some $n\geq 3$ and let $H$ be a group
such that $\Gamma_G\cong \Gamma_H$.
Suppose that $x$ is an element in $G$ such that
$Cyc_G(x)=\left<x\right>$.
\begin{enumerate}
\item  If $Cyc(G)\not=1$, then both $G$ and $H$ are isomorphic to
$Q_{2^n}$, the generalized quaternion group of order $2^n$. \item
If $|x|=p=2$, then $|G|=|H|$. This is the case, if $G$ has a
direct factor of order $2$. \item If $|x|\in\{p^{n-1},p^{n-2}\}$,
then $|G|=|H|$. This is the case, if $G$ is of exponent $p^{n-1}$
or $p^{n-2}$, or $G$ has a cyclic direct factor of order
$p^{n-2}$.
\end{enumerate}
\end{prop}
\begin{proof}
(1) \;  By Proposition \ref{Qn}, $G\cong Q_{2^n}$.  Thus
$\Gamma_H\cong \Gamma_{Q_{2^n}}$. We have $2^n-2=|H|-|Cyc(H)|$,
since $|Cyc(Q_{2^n})|=2$. As the exponent of $Q_{2^n}$ is
$2^{n-1}$, there exists an element $a\in Q_{2^n}$ such that
$D=Cyc_{Q_{2^n}}(a)=\left<a\right>$. It follows that
$Cyc_{D}(D)=D$. If $\phi:\Gamma_{Q_{2^n}}\rightarrow \Gamma_{H}$
is a graph isomorphism, then
$$|C\backslash Cyc_{C}\big(C\big)|=|D\backslash
Cyc_{D}\big(D\big)|,$$  where $C=Cyc_H(\phi(a))$. Hence it follows
from  Lemma \ref{12} that there exists an element $b\in
H\backslash Cyc(H)$ such that $Cyc_H(b)$ is a cyclic subgroup of
$H$ and $2^n-2^{n-1}=|H|-|Cyc_H(b)|$. Now  Proposition \ref{fi}
implies that $|Cyc(H)|$ must divide  $\gcd(2^{n-1},2^n-2)=2$.
Therefore $|Cyc(H)|\in\{1,2\}$. If $|Cyc(H)|=1$, then $|H|=2^n-1$
and $|Cyc_H(b)|=2^{n-1}-1$. Since $Cyc_H(b)\leq H$,  we  have that
$2^{n-1}-1$ divides $2^n-1$, which is impossible, since $n\geq 3$.
Thus $|Cyc(H)|=2$ and so $|H|=2^n$. Now Proposition
\ref{Qn} completes the proof of part (1).\\

(2) \; By part (1), we may assume that $Cyc(G)=1$. Then
$|H|-|Cyc(H)|=2^n-1$ and $|H|-|Cyc_H(h)|=2^n-2$ for some vertex
$h$ in $\Gamma_H$. Thus $|Cyc(H)|$ divides $\gcd(2^n-1,2^n-2)=1$
and so $|H|=|G|$. \\

(3) \; By part (1), we may assume that $Cyc(G)=1$.
 If $\phi:\Gamma_G\rightarrow \Gamma_H$ is a graph
isomorphism, then
$$|D_1\backslash Cyc_{D_1}\big(D_1\big)|=|D_2\backslash
Cyc_{D_2}\big(D_2\big)|,$$ where $D_1=Cyc_G(a)$,
$D_2=Cyc_H(\phi(a))$ and $a$ is any vertex of $\Gamma_G$. Assume
that $|x|=p^i$. It follows from Lemma \ref{12} that there exists a
vertex $s$ in $\Gamma_H$ such that $Cyc_H(s)$ is a cyclic
 subgroup of $H$,
$$|G|-|Cyc_G(x)|=|H|-|Cyc_H(s)| \;\;\text{and}\;\;  |G|-|Cyc(G)|=|H|-|Cyc(H)|.$$ Therefore
$|H|-|Cyc(H)|=p^n-1$, $|H|-|Cyc_H(s)|=p^n-p^i$.  It follows that
 $|Cyc(H)|$ divides $\gcd(p^i-1,p^n-1)=p^{d}-1$, where $d=gcd(i,n)$.
Thus there exists  an integer $\ell$ such that
$$p^d-1=|Cyc(H)|\ell. \eqno{\rm(*)}$$ Since
$\frac{|Cyc_H(s)|}{|Cyc(H)|}= \frac{p^{i}-1}{p^d-1}\ell+1$,
$\frac{|H|}{|Cyc(H)|}=\frac{p^{n}-1}{p^d-1}\ell+1$ and $Cyc_H(s)$
is a subgroup of $H$, we have that  $\frac{p^{i}-1}{p^d-1}\ell+1$
divides $\frac{p^{n}-1}{p^d-1}\ell+1$. It follows that
$\frac{p^{i}-1}{p^d-1}\ell+1$ divides
$\frac{p^{n}-1}{p^d-1}\ell+1-\frac{p^{i}-1}{p^d-1}\ell-1$,  so
$\frac{p^{i}-1}{p^d-1}\ell+1$ divides
$p^{i}\frac{p^{n-i}-1}{p^d-1}\ell$. Now since
$\gcd\big(\frac{p^{i}-1}{p^d-1}\ell+1,\ell\big)=1,$  we have
  $\frac{p^{i}-1}{p^d-1}\ell+1$ divides $p^{i}\frac{p^{n-i}-1}{p^d-1}$.
We must prove that $Cyc(H)=1$ and so by $(*)$,  it is
enough to prove $\ell=p^d-1$.\\

If $i=n-1$, then $d=1$ and so we have that
$\frac{p^{n-1}-1}{p-1}\ell+1$ divides $p^{n-1}$. This implies that
$\ell=p-1$, as required.\\

If $i=n-2$, then $d\in\{1,2\}$. If $d=1$, then
$\frac{p^{n-2}-1}{p-1}\ell+1$ divides $p^{n-2}(p+1)$. Now suppose,
for a contradiction, that $\ell<p-1$. It follows that
$\gcd\big(\frac{p^{n-2}-1}{p-1}\ell+1,p\big)=1$. Therefore
$\frac{p^{n-2}-1}{p-1}\ell+1$ divides $p+1$, which implies that
$n-2<1$, contrary to $n\geq 3$. Hence we have that $\ell=p-1$, as
required. Now assume that $d=2$. Therefore $n$ is even and we have
$\frac{p^{n-2}-1}{p^2-1}\ell+1$ divides $p^{n-2}$. Since $p^2\leq
\frac{p^{n-2}-1}{p^2-1}\ell+1$ and $3\leq n$ is even, we have that
$p^2$ must divide $\frac{p^{n-2}-1}{p^2-1}\ell+1$ which implies
that $p^2 \mid \ell+1$. Hence $\ell+1=p^2$, as
required.\\
\end{proof}
\section{\bf Groups whose non-cyclic graphs are unique}
In this section we study Question \ref{mainqu} for the group
property of being isomorphic to a fixed group. As we mentioned in
Remark \ref{rem-is}(2), in general, Question \ref{mainqu} is not
true for this property. But  we saw in Section \ref{a} that there
are groups $G$ whose non-cyclic graphs are ``unique'', that is, if
$\Gamma_G \cong \Gamma_H$ for some group $H$, then $G\cong H$.
Here we give some other groups whose non-cyclic graphs are unique.
Before this, we need to state some notations.

For any  finite group  $G$, we denote by $\pi_e(G)$ the set of
orders of elements of $G$. There is a uniquely characterized
subset  $\mu(G)$ of $\pi_e(G)$ with the following properties:
\begin{enumerate}
\item For every element $x$ of $G$, there exists $t\in \mu(G)$
such that $|x|$ divides $t$. \item If $t,s$ are two distinct
elements of $\mu(G)$,   then $t$ does not divide $s$ and vise
versa.
\end{enumerate}
It is clear that $$\pi_e(G)=\{s\in\mathbb{N} \;|\; s
\;\text{divides some member of}\; \mu(G)\}.$$ In fact $\mu(G)$ is
the set of elements of $\pi_e(G)$ which are maximal under the
divisibility relation.
\begin{lem}\label{mu}
 If $G$ is a
non-cyclic finite group, then $$\mu(G) \cap
\pi_e(Cyc(G))=\varnothing.$$
\end{lem}
\begin{proof}
Suppose, for a contradiction, that $n\in \mu(G) \cap
\pi_e(Cyc(G))$. It follows that there  exists an element $g\in
Cyc(G)$ such that $|g|=n$.  Since $G$ is not cyclic, there exists
an element $x\in G\backslash Cyc(G)$. Thus
$\left<x,g\right>=\left<y\right>$ for some $y\in G$ such that
$|y|>n$. Now by the definition of $\mu(G)$, there is  $m\in\mu(G)$
such that $|y|$  divides $m$. Therefore $n|m$ and since $n,m\in
\mu(G)$, we have that $n=m$. Hence $|y|=n$, a contradiction. This
completes the proof.
\end{proof}
\begin{lem}\label{mu2}
Let $g$ be an element of a finite group $G$ such that
$|g|\in\mu(G)$. Then $Cyc_G(g)=\left<g\right>$.
\end{lem}
\begin{proof}
Suppose, for a contradiction, that $x\in
Cyc_G(g)\backslash\left<g\right>$. Then
$$\left<g\right>\varsubsetneqq\left<g,x\right>=\left<y\right>$$ for
some $y\in G$. It follows that $|g|<|y|$. On the other hand, by
the definition of the set $\mu(G)$, there exists  $s\in\mu(G)$
such that $|y|$ divides $s$. Therefore $|g|$ divides $s$, so
$|g|=s$. It follows that $|g|=|y|$, a contradiction. This
completes the proof.
\end{proof}
\begin{thm}\label{conj}
Let $G$ and $H$ be two finite non-cyclic groups such that
$\Gamma_G\cong\Gamma_H$. If $|G|=|H|$, then $\pi_e(G)=\pi_e(H)$.
\end{thm}
\begin{proof}
Let $t\in \mu(G)$. Then by Lemma \ref{mu}, there exists an element
$g\in G\backslash Cyc(G)$ such that $|g|=t$, and by Lemma
\ref{mu2}, $Cyc_G(g)=\left<g\right>$. Since $\Gamma_G\cong
\Gamma_H$, there exists an element $h\in H\backslash Cyc(H)$ such
that $Cyc_D(D)=D$ where $D=Cyc_H(h)$. Now  Lemma \ref{12} implies
that $Cyc_H(h)$ is a cyclic subgroup. Also since $\Gamma_G\cong
\Gamma_H$ and $|G|=|H|$, we have that $|Cyc_G(g)|=|Cyc_H(h)|$
from which it follows that $H$ has an element of order $t$. Hence
we have proved that $\mu(G)\subseteq \pi_e(H)$ which implies that
$\pi_e(G)\subseteq \pi_e(H)$. By the symmetry between $G$ and
$H$, we have that $\pi_e(H)\subseteq \pi_e(G)$ , completing the
proof.
\end{proof}
Now we give some groups with unique non-cyclic graphs. For an
integer $n>2$, we denote by   $D_{2n}$ the dihedral group of order
$2n$.
\begin{prop} \label{QD} Let $n>2$ be an integer.
 If $G$ is a group with $\Gamma_G \cong \Gamma_{D_{2n}}$, then
$|G|=2n$ and $G$ contains a cyclic subgroup of order $n$. In
particular, if $n$ is odd, then $G\cong D_{2n}$.
\end{prop}
\begin{proof}
We have $2n-1=|G|-|Cyc(G)|$, since $Cyc(D_{2n})=1$. The dihedral
group $D_{2n}$ has   an element $a$ of order $n$, such that
$D=Cyc_{D_{2n}}(a)=\left<a\right>$. It follows that
$Cyc_{D}(D)=D$. It follows from the hypothesis and  Lemma
\ref{12} that  there exists an element $x\in G\backslash Cyc(G)$
such that $Cyc_G(x)$ is a cyclic subgroup of $G$ and
$2n-n=|G|-|Cyc_G(x)|$. Now  by Proposition \ref{fi}, that
$|Cyc(G)|$  divides $\gcd(2n-1,n)=1$. Hence $|G|=2n$ and so by
Theorem \ref{conj}, $G$ contains an element of order $n$. Now
assume that $n$ is odd. By Theorem \ref{conj}, $\mu(G)=\{n,2\}$.
Since $|G|=2n$ and $n$ is odd, it follows that
$C_G(b)=\left<b\right>$ for every involution $b\in G$. Let $A$ be
a cyclic subgroup of order $n$ in $G$ which is clearly normal in
$G$. Thus $G=A\left<b\right>$ for every involution $b\in G$. On
the other hand, since $C_G(b)=\left<b\right>$, $b$ acts fixed
point freely on $A$, thus $c^b=c^{-1}$ for all $c\in A$. This
implies that $G\cong D_{2n}$, as required.
\end{proof}
Thus the non-cyclic graphs of  generalized quaternion groups are
unique. What about the other finite $p$-groups with a maximal
cyclic subgroup? We know the complete classification of these
groups (see \cite[Theorem 5.3.4]{Rob}) and in the following  we
show that their non-cyclic graphs are not unique, in general, but
Question \ref{mainqu} has positive answer when $\mathcal{P}$ is
the property of being a finite $p$-group for a fixed prime $p$ of
fixed size with a maximal cyclic subgroup.
\begin{prop}
Let  $p$ be a prime number and $n\geq 3$, $m\geq 4$ be integers.
Suppose that
$G(p^n)=\left<a,x\;|\; x^p=a^{p^{n-1}}=1, a^x=a^{1+p^{n-2}}\right>$,\\
$H=\left<a,x\;|\; x^2=a^{2^{m-1}}=1, a^x=a^{2^{m-2}-1}\right>$ and
$K(p^n)=\mathbb{Z}_{p^{n-1}}\oplus \mathbb{Z}_p$.
\begin{enumerate}
\item If $S$ is a group such that $\Gamma_S\cong \Gamma_{K(8)}$,
then $S\cong K(8)$. \item If either $n>3$ or $p>2$, then
$\Gamma_{G(p^n)}\cong \Gamma_{K(p^n)}$. \item If $S$ is a group
such that $\Gamma_S\cong \Gamma_{H}$, then $S\cong H$. \item If
$S$ is a group such that $\Gamma_S\cong \Gamma_{D_{2^n}}$, then
$S\cong D_{2^n}$. \item Let $S$ be a group. Then $\Gamma_S\cong
\Gamma_{K(p^n)}$ if and only if $S\cong G(p^n)$ or $K(p^n)$.
\end{enumerate}
\end{prop}
\begin{proof}
(1) \; Since $Cyc_{K(8)}((1,0))=\left<(1,0)\right>$ and
$Cyc_{K(8)}((0,1))=\left<(0,1)\right>$, it follows from
Proposition \ref{p-p^2}, that $|S|=8$. Now Propositions
\ref{p-p^2}-(1) and \ref{p^3} imply that $S\cong D_8$ or $S\cong
K(8)$. But $\Gamma_{D_8}\not\cong\Gamma_{K(8)}$, since $(2,0)$ has
degree 2 in $\Gamma_{K(8)}$ while every vertex in $\Gamma_{D_8}$
has degree  equal to $6$ or $4$. This completes the proof of part (1).\\
(2) \; It is easy to see that if  $\gcd(j,p)=1$, then
$Cyc_{G(p^n)}(a^ix^j)=\left<a^ix^j\right>$, and if $p^{n-1} \nmid
i$, then
$$Cyc_{G(p^n)}(a^i)=\big\{bc \;|\;
b\in\left<a\right>,|b|>|a^i|,c\in\left<x\right>\backslash\{1\}\big\}\cup
\left<a\right>.
$$
Also if  $\gcd(j,p)=1$, then
$Cyc_{K(p^n)}\big((i,j)\big)=\left<(i,j)\right>$, and if $p^{n-1}
\nmid i$, then
\begin{align*}
Cyc_{K(p^n)}\big((i,0)\big)=\big\{(b,c) \;:\;
b\in\mathbb{Z}_{p^{n-1}},|b|>|i|,c\in\mathbb{Z}_p\backslash\{0\}\big\}\bigcup
\left<(1,0)\right>.
\end{align*}
Now using these information, it is easy to see that
$\phi:V(\Gamma_{G(p^n)})\rightarrow V(\Gamma_{K(p^n)})$ defined by
$\phi(a^ix^j)=(i,j)$ is a graph
isomorphism from $\Gamma_{G(p^n)}$ to $\Gamma_{K(p^n)}$.\\
(3) \; We have that
$$Cyc_H(a^ix^j)=\begin{cases}\left<a^ix\right>& \text{if}\; j=1\\
\left<a\right> & \text{if}\; j=0 \;\text{and}\; 2^{m-2}\nmid i\\
\{a^jx\;:\; 2\nmid j\} \cup \left<a\right> & \text{if} \; j=0\;,
2^{m-1}\nmid i \;\text{and}\; 2^{m-2}\mid i
\end{cases}.$$
Thus, it follows from Proposition \ref{p-p^2} that $|S|=|H|=2^m$.
Now Theorem \ref{conj} implies that $S$ has an element of order
$2^{m-1}$. Therefore \cite[Theorem 5.3.4]{Rob} and Proposition
\ref{p-p^2}-(1) yield that $S\cong H$,  $D_{2^m}$, $K(2^m)$ or
$G(2^m)$. But $\Gamma_{H}\not\cong \Gamma_{K(2^m)}$, since the set
of degrees of  vertices in $\Gamma_{K(2^m)}$ is $\{2,2^m-2^j \;|\;
j\in\{1,\dots,m-1\}\}$, while the corresponding set in
$\Gamma_{H}$ is $\{2,4,2^{m-2}+2^{m-1}\}$. So it follows from part
(2) that $S\cong H$ or $S\cong D_{2^m}$. Now  for
$D_{2^m}=\left<b,y \;|\; b^{2^{m-1}}=y^2=1, b^y=b^{-1}\right>$, we
have
$$Cyc_H(b^iy^j)=\begin{cases}\left<b^iy\right>& \text{if}\; j=1\\
\left<b\right> & \text{if}\; j=0 \;\text{and}\; 2^{m-1}\nmid i
\end{cases}.$$
This implies that $\Gamma_H \not\cong \Gamma_{D_{2^m}}$, since the
set of degrees of vertices in $\Gamma_{D_{2^m}}$ is
$\{2,2^{m-1}\}$. Hence $S\cong H$, as required.\\
(4) \; By Proposition \ref{QD}, $|S|=2^n$ and $S$ contains an
element of order $2^{n-1}$. Now if $n>3$,  it follows from Theorem
\cite[Theorem 5.3.4]{Rob}, Proposition \ref{p-p^2}-(1) and the
proof of part (3) that  $S\cong D_{2^n}$. If $n=3$, then
Propositions \ref{p-p^2}-(1) and \ref{p^3} imply  that $S\cong
D_8$ or $S\cong
K(8)$. Now part (1) completes the proof.\\
(5) \; It follows from parts (1)-(4), Proposition \ref{p-p^2} and
\cite[Theorem 5.3.4]{Rob}.
\end{proof}

There is a conjecture due to Shi and Bi (see Conjecture 1 of
\cite{ShiBi91})  saying that if $M$ is a finite non-abelian simple
group such that $|G|=|M|$ for some group $G$ with
$\pi_e(G)=\pi_e(M)$, then $G\cong M$. This conjecture has been
proved for many non-abelian finite simple groups.
\begin{thm}{\rm (See
\cite{CShi,Shi94,Shi87,ShiBi91,ShiBi92,ShiBi90,XUShi})}\label{shi}
Let $G$ be a finite group and $M$ one of the following finite
simple groups: {\rm (1)} A cyclic simple group $\mathbb{Z}_p$;
{\rm (2)} An alternating group $A_n$, $n\geq 5$; {\rm (3)} A
sporadic simple group; {\rm (4)} A Lie type group except $B_n$,
$C_n$, $D_n$ {\rm(}$n$ even{\rm)}; or {\rm (5)} A simple group
with order $<10^8$. Then $G\cong  M$ if and only if {\rm (a)}
$\pi_e(G)=\pi_e(M)$, and {\rm (b)} $|G|=|M|$.
\end{thm}
This conjecture  has been also proved  for some non-simple
groups. The following is an example.
\begin{thm}{\rm (See \cite{Bi})}\label{BI} Let $G$ be a group. Then $G\cong  S_n$, $n\geq  3$
{\rm (}$S_n$ denotes the symmetric group of degree $n${\rm )} if
and only if {\rm (a)} $\pi_e(G)=\pi_e(S_n)$, and {\rm (b)}
$|G|=|S_n|$.
\end{thm}
\begin{thm}
Let $G$ be a finite simple sporadic group. If $\Gamma_G\cong
\Gamma_H$ for some group $H$, then $G\cong H$.
\end{thm}
\begin{proof}
We first prove that $|G|=|H|$. It is easy to see (e.g., from Table
III in \cite{Chen} or Tables 1a-1c in \cite{Maz}) that there is a
prime divisor $p$ of $|G|$ such that $$p-1 \;\text{divides}\;\;
|G|\eqno{(I)}$$ and $C_G(g)=\left<g\right>$ for every element
$g\in G$ of order $p$. Since $\Gamma_G \cong \Gamma_H$,
$|G|-1=|H|-|Cyc(H)|$. It follows that
$$|Cyc(H)| \;\;\text{divides}\;\; |G|-1.\eqno{(*)}$$ Now if $g\in
G$ and $|g|=p$, then there exists an element $h\in H\backslash
Cyc(H)$ such that $|G|-|Cyc_G(g)|=|H|-|Cyc_H(h)|$. But
$C_G(g)=\left<g\right>$ implies that $Cyc_G(g)=\left<g\right>$.
Thus $|G|-p=|H|-|Cyc_H(h)|$. Now Proposition \ref{fi} yields that
$$|Cyc(H)| \;\; \text{divides}\;\; |G|-p. \eqno{(**)}$$ Now it
follows from $(*)$ and $(**)$ that $|Cyc(H)|$  divides $p-1$ and
so by $(I)$ it must divide $|G|$. Therefore $|Cyc(H)|$ divides
$\gcd(|G|,|G|-1)=1$, and from this it follows that $|G|=|H|$.
Hence by Theorem \ref{conj}, we have  $\pi_e(G)=\pi_e(H)$. Now
Theorem \ref{shi}-(3), completes the proof.
\end{proof}
\begin{thm}
If $G$ is a group and $n>2$ is an integer, then the following hold:\\
{\rm 1)} If $\Gamma_{G}\cong \Gamma_{S_n}$, then $G\cong S_n$.\\
{\rm 2)} If $n>3$ and $\Gamma_G\cong \Gamma_{A_n}$, then $G\cong
A_n$.
\end{thm}
\begin{proof}
1)\; Let $a$ and $b$ be the cycles $(1,2,\dots,n)$ and
$(1,2,\dots,n-1)$, respectively. Then $C_{S_n}(a)=\left<a\right>$
and $C_{S_n}(b)=\left<b\right>$. Thus
$Cyc_{S_n}(a)=\left<a\right>$ and $Cyc_{S_n}(b)=\left<b\right>$.
Now since $\Gamma_G\cong \Gamma_{S_n}$, we have that  $|Cyc(G)|$
divides $n-(n-1)=1$. Thus $|Cyc(G)|=1$ and so $|G|=|S_n|$. Now
Theorems \ref{conj} and \ref{BI} implies that $G\cong S_n$.\\
 2)\;
Suppose first that $n$ is odd. Then $a\in A_n$ and $C_{A_n}(a)=
\left<a\right>$. Thus $Cyc_{A_n}(a)=\left<a\right>$. Since
$\Gamma_G\cong \Gamma_{A_n}$, we have  $|Cyc(G)|$ divides
$\gcd(n-1,\frac{n!}{2}-1)=1$, since $n>3$.
 Hence $Cyc(G)=1$ and so, in this case, $|G|=|A_n|$.\\
Now assume that $n$ is even. Then $b\in A_n$ and $C_{A_n}(b)=
\left<b\right>$ and so $Cyc_{A_n}(b)=\left<b\right>$. It follows
that  $|Cyc(G)|$ divides $\gcd(n-2,\frac{n!}{2}-1)=1$, since
$n>3$. Therefore $Cyc(G)=1$ and so $|G|=|A_n|$. Therefore in any
case, $|G|=|A_n|$ and so by  Theorem \ref{conj} we have
$\pi_e(G)=\pi_e(A_n)$. Therefore  Theorem  \ref{shi}-(2) implies
that $G\cong A_n$ for $n\geq 5$ and since $A_4$ is the only group
of order 12 having no  element of order 6,  for $n=4$ we have also
$G\cong A_n$. This completes the proof.
\end{proof}
Let $G$ be a finite non-trivial group. The Gruenberg-Kegel graph
(or prime graph) of  $G$  is the graph whose vertices are prime
divisors of $|G|$, and two distinct primes $p$ and $q$ are
adjacent if $G$ contains an element of order $pq$. Denote by
$s(G)$ the number of connected components of the Gruenberg-Kegel
graph of $G$. The finite simple groups with non-connected
Gruenberg-Kegel graph were classified by Williams \cite{W} and
Kondrat'ev \cite{Kont}. A list of these groups  can be found in
\cite{Maz}.
\begin{thm}\label{tom}
Let $M$ be a finite non-abelian simple group of Lie type with
$s(M)\geq 2$. If $\Gamma_G\cong \Gamma_M$ for some group $G$, then
$|G|=|M|$.
\end{thm}
\begin{proof}
Note that if $T$ is a cyclic subgroup of a group $H$ such that
$C_H(T)=T$, then $Cyc_H(T)=T$. Now the proof is {\em mutatis
mutandis} the proof of \cite[Theorem 3]{moghadam}, except that
instead  of Lemma 2(a) of \cite{moghadam}, we may use
Proposition  \ref{fi}.
\end{proof}
\begin{thm} Let $M$ be a finite non-abelian simple group of Lie
type except $B_n$, $C_n$, $D_n$ {\rm(}$n$ even{\rm )} with
$s(M)\geq 2$. If $\Gamma_M\cong \Gamma_G$ for some group $G$, then
$G\cong M$.
\end{thm}
\begin{proof}
It follows from Theorems \ref{tom}, \ref{conj} and \ref{shi}-(4).
\end{proof}
\begin{rem}{\rm
 Professor V.D. Mazurov told the first author in
Antalya Algebra Days VII (May 2005),  that the validity of Shi-Bi
conjecture will be completed within the next two years  for all
finite non-abelian simple groups, and so far its validity for all
finite simple groups except those  of Lie types $B_n$ and $C_n$
have been proved. So in view of Theorems \ref{tom} and \ref{conj},
it will be possible to say that the non-cyclic graph of a
non-abelian finite simple group is unique. Anyway we close this
paper by putting  forward the following conjecture.}
\end{rem}
\begin{con}
Let $M$ be a finite non-abelian simple group. If
$\Gamma_M\cong\Gamma_G$ for some group $G$, then $G\cong M$.
\end{con}
\noindent {\bf Acknowledgments.}  The authors were supported by
Isfahan University Grant no. 830819 and its Center of Excellence
for Mathematics.


\begin{thebibliography}{99}
\bibitem{AKM} A. Abdollahi, S. Akbari and H.R. Maimani, {\sl Non-commuting graph of a
group}, J. Algebra {\bf 298} (2006), 468-492..
\bibitem{BHM} E. A. Bertram, M. Herzog and A. Mann,
{\sl On a graph related to conjugacy classes of groups,} Bull.
London Math. Soc. {\bf 22} (1990), no. 6, 569--575.
\bibitem{Bi} J. Bi, {\sl A characterization of
symmetric groups (Chinese)}, Acta. Math. Sinica, {\bf 33} (1990),
70-77.
\bibitem{CShi} H. Cao and W. Shi, {\sl  Pure quantitative characterization of finite
projective special unitary groups},  Sci. China Ser.A {\bf 45}
(2002), no. 6, 761-772.
\bibitem{Chen} G. Chen, {\sl On Thompson's conjecture,} J.
Algebra, {\bf 185} (1996), 184-193.
\bibitem{Dies} R. Diestel,  Graph theory, Second edition. GTM 173, Springer-Verlag, New York, 2000.
\bibitem{goor} R. Goormaghtigh, {\sl L'interm\'ediaire des math\'ematiciens}, {\bf 24} (1917), 88.
\bibitem{GKNP} F. Grunewald, B. Kunyavski\u\i, D.  Nikolova and  E. Plotkin,
{\sl Two-variable identities in groups and Lie algebras,} J. Math.
Sci. (New York) {\bf 116} (2003), no. 1, 2972--2981.
\bibitem{Isac} I. M. Isaacs, {\sl Equally partitioned groups,} Pacific J. Math. {\bf
49} (1973), 109--116.
 \bibitem{Kont} A. S. Kondrat'ev, {\sl On prime graph
components of finite simple groups}, Mat. Sb. {\bf 180} (1989),
787-797.
\bibitem{Le} M. Le, {\sl Exceptional
solutions of the exponential Diophantine equation $(x\sp
3-1)/(x-1)=(y\sp n-1)/(y-1)$}, J. Reine Angew. Math. {\bf 543}
(2002), 187--192.
\bibitem{Maz} V. D. Mazurov, {\sl Recognition of
the finite simple groups $S\sb 4(q)$ by their element orders},
Algebra Logic {\bf 41} (2002), no. 2, 93--110.
\bibitem{moghadam} A. R. Moghaddamfar, W. J. Shi, W. Zhou and A. R. Zokayi, {\sl On the noncommuting graph
associated with a finite group}, Siberian Math. J.  {\bf 46}
(2005) no. 2, 325--332.
\bibitem{N} B. H. Neumann,  {\sl A problem of Paul Erd\"os on
groups,}  J. Austral. Math. Soc. Ser. A {\bf 21} (1976), 467-472.
\bibitem{Holms2} K. O'Bryant, D. Patrick, L. Smithline and E.
Wepsic, {\sl Some facts about cycles and tidy groups,} Rose-Hulman
Institute of Technology, Indiana, USA, Technical Report MS-TR
92-04, (1992).
\bibitem{Holms1} D. Patrick and E. Wepsic, {\sl Cyclicizers, centralizers and
normalizers,}  Rose-Hulman Institute of Technology, Indiana, USA,
 Technical Reprot MS-TR 91-05, (1991).
\bibitem{P} L. Pyber, {\sl The number of pairwise
noncommuting elements and the index of the centre in a finite
group,} J. London Math. Soc. (2) {\bf 35} (1987), no. 2, 287--295.
\bibitem{Rob} Derek J. S. Robinson, A course in the theory of
groups, GTM 80, Springer-Verlag, New York, 1982.
\bibitem{Shi94} W. Shi,
{\sl Pure quantitative characterization of finite simple groups
I.},  Progr. Natur. Sci. (English Ed.) {\bf 4} (1994), no. 3,
316--326.
\bibitem{Shi87} W. Shi, {\sl A new characterization
of the sporadic simple groups}, Group theory (Singapore, 1987),
531-540, de Gruyter, Berlin, 1989.
\bibitem{ShiBi91} W. Shi and J. Bi, {\sl A characteristic
property for each finite projective special linear group},
Groups--Canberra 1989, 171–180, Lecture Notes in Math., {\bf
1456}, Springer, Berlin, 1990.
\bibitem{ShiBi92} W. Shi and J. Bi, {\sl A new characterization of the alternating
groups}, Southeast Asian Bull. Math. {\bf 16} (1992), no. 1,
81--90.
\bibitem{ShiBi90} W. J. Shi and J. Bi, {\sl A characterization
of Suzuki-Ree groups}, Science in China (Ser. A), {\bf 34} (1991),
14-19.
\bibitem{Tom} M. J. Tomkinson, {\sl Groups covered by finitely many cosets or
subgroups}, Comm. Algebra {\bf 15} (1987), no. 4, 845--859.
\bibitem{W} J. S. Williams, {\sl Prime graph components of finite groups,} J.
Algebra {\bf 69} (1981), no. 2, 487--513.
\bibitem{XUShi} M. Xu and
W. Shi, {\sl   Pure quantitative characterization of finite simple
groups $~^2D_n(q)$ and $D_l(q)$ ($l$ odd)}, Algebra Colloq. {\bf
10} (2003), no. 3, 427--443.
\end{thebibliography}
\end{document}